\pdfoutput=1
\documentclass[10pt]{amsart}
\usepackage{amscd}
\usepackage{amssymb}
\date{11 November 2026}
\title{The Simplicial Cylinder DG Ring}

\usepackage{libertinus}
\usepackage[T1]{fontenc}
\usepackage[scaled=0.95]{roboto}
\usepackage[varbb, varvw, smallerops]{newtxmath}
\usepackage[cal=cm, calscaled=0.95, frak=euler, frakscaled=0.98]{mathalfa}
\usepackage[scaled=0.95]{inconsolata}

\usepackage{tikz}
\usepackage{tikz-cd}
\usetikzlibrary{arrows}
\tikzcdset{arrow style=tikz,
diagrams={>={Stealth[round,length=4pt,width=6pt,inset=3pt]}}}
\tikzcdset{every label/.append style = {font = \footnotesize}}
\usepackage{graphicx}
\usepackage{hyperref}
\hypersetup{colorlinks=false}

\usepackage{comment}
\hypersetup{colorlinks=false}

\author{Amnon Yekutieli}
\address{Department of  Mathematics,
Ben Gurion University, Be'er Sheva 84105, Israel.}
\email{\href{mailto:amyekut@gmail.com}{amyekut@gmail.com}}
\urladdr{\url{https://sites.google.com/view/amyekut-math/home}}


\newtheorem{thm}[equation]{Theorem}
\newtheorem{cor}[equation]{Corollary}
\newtheorem{prop}[equation]{Proposition}
\newtheorem{lem}[equation]{Lemma}
\theoremstyle{definition}
\newtheorem{dfn}[equation]{Definition}
\newtheorem{rem}[equation]{Remark}
\newtheorem{exa}[equation]{Example}

\newtheorem{conv}[equation]{Convention}

\numberwithin{equation}{section}

\newcommand{\iso}{\xrightarrow{%
\smash{\raisebox{-0.5ex}{\ensuremath{\scriptstyle \simeq  \mspace{2mu}}}}}}

\newcommand{\xar}{\xrightarrow}
\newcommand{\opn}{\operatorname}
\newcommand{\cat}[1]{\operatorname{\mathsf{#1}}}

\newcommand{\cd}{\,{\cdot}\,}

\newcommand{\rmitem}[1]{\item[\text{\textup{(#1)}}]}

\newcommand{\mrm}[1]{\mathrm{#1}}

\newcommand{\La}{\Lambda}

\newcommand{\si}{\sigma}

\newcommand{\al}{\alpha}
\newcommand{\be}{\beta}
\newcommand{\ga}{\gamma}

\newcommand{\ze}{\zeta}

\newcommand{\K}{\mathbb{K}}
\newcommand{\Z}{\mathbb{Z}}
\newcommand{\N}{\mathbb{N}}
\newcommand{\bi}{\boldsymbol{i}}
\newcommand{\bj}{\bsym{j}}

\newcommand{\bpi}{\boldsymbol{\pi}}
\newcommand{\Simp}{\boldsymbol{\Delta}}

\newcommand{\tup}[1]{\textup{#1}}
\newcommand{\bsym}[1]{\boldsymbol{#1}}
\newcommand{\boplus}{\bigoplus\nolimits}
\newcommand{\ot}{\otimes}

\newcommand{\til}[1]{\tilde{#1}}

\newcommand{\bra}[1]{\langle #1 \rangle}
\renewcommand{\d}{\mathrm{d}}
\newcommand{\pa}{\partial}

\newcommand{\lb}{\linebreak}
\newcommand{\sbmat}[1]{\left[ \begin{smallmatrix} #1
\end{smallmatrix} \right]}
\newcommand{\bmat}[1]{\begin{bmatrix} #1 \end{bmatrix}}
\newcommand{\centover}{\mspace{0.0mu} /_{\mspace{-2mu}
\mrm{c}}\mspace{1.5mu}}
\renewcommand{\over}{\mspace{0mu} / \mspace{-0.5mu}}

\newcommand{\sub}{\subseteq}
\newcommand{\twoto}{\Rightarrow}
\newcommand{\colim}[1]{\underset{#1}{\opn{colim}} \mspace{3mu}}
\newcommand{\mylim}[1]{\displaystyle\lim_{#1} \mspace{3mu}}

\newcommand{\lmsp}{\mspace{1.5mu}}

\newcommand{\bmsp}{\mspace{10mu}}

\begin{document}

\begin{abstract}
Let $A$ and $B$ be DG rings, and let $f_0, f_1 : A \to B$ be DG ring
homomorphisms. The appropriate notion of {\em homotopy}
$\ga : f_0 \twoto f_1$ was introduced by Keller in 1999.
These homotopies are encoded as DG ring homomorphisms from $A$
to the {\em Keller cylinder DG ring} $\opn{Cyl}_{\mrm{Ke}}(B)$.

Recently we discovered that for every $q \in \N$ there is a DG ring
$\opn{Cyl}_q(B)$, which for $q = 1$ equals $\opn{Cyl}_{\mrm{Ke}}(B)$.
The collection
$\opn{Cyl}(B) := \{ \opn{Cyl}(B)_{q} \}_{q \in \N}$ is a {\em simplicial DG
ring}, which we call the {\em simplicial cylinder DG ring} of $B$.
Let us denote by $\opn{SHom}_q(A, B)$
the set of DG ring homomorphisms $A \to \opn{Cyl}_{q}(B)$.
As $q$ varies, we obtain the {\em simplicial set}
$\opn{SHom}(A, B) := \{ \opn{SHom}_q(A, B) \}_{q \in \N}$.

Now assume $\til{A}$ is a {\em semi-free DG ring}.
Our main theorem states that the simplicial
set $\opn{SHom}(\til{A}, B)$ is a {\em Kan complex}.
Let us denote by $\opn{SHom}_{\leq 1}(\til{A}, B)$ the {\em
fundamental groupoid} of this Kan complex.
Its objects are the DG ring homomorphisms $f : \til{A} \to B$,
and its isomorphisms are the simplicial homotopy classes $[\ga]$ of Keller
homotopies $\ga : f_0 \twoto f_1$.

Another theorem states that a surjective DG ring quasi-isomorphism
$\til{B} \to B$ induces a surjective equivalence of groupoids
$\opn{SHom}_{\leq 1}(\til{A}, \til{B}) \to
\opn{SHom}_{\leq 1}(\til{A}, B)$.
It follows that given a DG ring homomorphism $f : A \to B$, and semi-free
resolutions $\til{A} \to A$ and $\til{B} \to B$,
there is a {\em distinguished homotopy class} $[\ga]$
between any two DG ring liftings
$\til{f}_0, \til{f}_1 : \til{A} \to \til{B}$ of $f$.

Using the results of this paper, we expect to produce
an explicit construction of the $(2, 1)$-derived category of the category
$\cat{DGRng}$ of DG rings, in which the Hom groupoids are  
$\opn{SHom}_{\leq 1}(\til{A}, \til{B})$. 

The category $\cat{DGRng}$ is a full subcategory of the category
$\cat{DGCat}$ of (small)  {\em DG categories}.
Presumably all the results above can be extended from
$\cat{DGRng}$ to $\cat{DGCat}$.
\end{abstract}

\maketitle


\tableofcontents
\mbox{}

\setcounter{section}{-1}
\section{Introduction}

By {\em DG ring} we mean a graded ring
$A = \bigoplus_{i \in \Z} A^i$, with a differential satisfying the
graded Leibniz rule. (Traditionally $A$ is called an associative unital
cochain DG algebra.)
The category of DG rings is denoted by $\cat{DGRng}$.

A DG ring $A$ is the same as a single object DG category $\cat{A}$,
and a DG functor $F : \cat{A} \to \cat{B}$ between single object DG categories
is a DG ring homomorphism $f : A \to B$. Thus $\cat{DGRng}$ is a full
subcategory of $\cat{DGCat}$, the category of (small) DG categories.
For the sake of streamlining the discussion, in this paper we only look at DG
rings. See Remark \ref{rem:248} regarding the potential extension of this work
to DG categories.

Let $A$ and $B$ be DG rings, and let $f_0, f_1 : A \to B$ be DG ring
homomorphisms. A {\em Keller homotopy}
$\ga : f_0 \twoto f_1$ is a degree $-1$ homomorphism of graded abelian groups
$\ga : A \to B$, which is a homotopy on the underlying DG abelian groups, and
is also an $f_0$-$f_1$-derivation.
This concept was introduced in \cite{Ke1}, and
we recall it in Definition \ref{dfn:330}.

It is known that when the DG ring $A$ is semi-free (see Definition
\ref{dfn:180}),
the Keller homotopies form an equivalence relation on the set
$\opn{Hom}_{\cat{DGRng}}(A, B)$. But for an arbitrary DG ring $A$ that might
fail, according to \cite[Correction]{Ke1} and \cite{Ke3}.

The {\em Keller cylinder DG ring} $\opn{Cyl}_{\mrm{Ke}}(B)$ is recalled
in Definitions \ref{dfn:332}. Keller homotopies $\ga : f_0 \twoto f_1$
are encoded by DG ring homomorphisms $g : A \to \opn{Cyl}_{\mrm{Ke}}(B)$.
These ideas were also introduced in the paper \cite{Ke1}, and we review them
in Section \ref{sec:prelims-dg} of the paper.
Some texts, including \cite{Ta}, refer to $\opn{Cyl}_{\mrm{Ke}}(B)$
as the {\em path object} associated to $B$.

Recently we discovered that the Keller cylinder DG ring
$\opn{Cyl}_{\mrm{Ke}}(B)$ is the special case $q = 1$ of a collection of DG
rings $\opn{Cyl}_q(B)$, $q \in \N$. These are naturally
arranged as a simplicial DG ring
$\opn{Cyl}(B) := \{ \opn{Cyl}_{q}(B) \}_{q \in \N}$,
called the {\em simplicial cylinder DG ring of $B$}.
This construction is in Definition \ref{dfn:215} of the paper.

Given DG rings $A$ and $B$, for every $q$ there is the set
\begin{equation} \label{eqn:285}
\opn{SHom}_{q}(A, B) :=
\opn{Hom}_{\cat{DGRng}}(A, \opn{Cyl}_{q}(B)) .
\end{equation}
As $q$ changes, this yields the simplicial set
\begin{equation} \label{eqn:286}
\opn{SHom}_{}(A, B) :=
\bigl\{ \opn{SHom}_{q}(A, B) \bigr\}_{q \in \N} ,
\end{equation}
called the {\em simplicial cylinder Hom set} from $A$ to $B$.
The simplicial set
$\opn{SHom}_{}(A, B)$ is functorial in $A$ and $B$.

The main result of our paper is this:

\begin{thm} \label{thm:150}
Let $\til{A}$ be a semi-free DG ring, and let $B$ be any DG ring.
Then the simplicial set
$\opn{SHom}_{}(\til{A}, B)$
is a Kan complex.
\end{thm}

The theorem is repeated as Theorem \ref{thm:215} in Section
\ref{sec:kan-cond}, and proved there.
Theorem \ref{thm:150} shows that the equivalence relation on the set
$\opn{Hom}_{\cat{DGRng}}(\til{A}, B) = \opn{SHom}_{0}(\til{A}, B)$
formed by the Keller derivations, which was mentioned above, consists of
isomorphisms in the {\em fundamental groupoid}
\begin{equation} \label{eqn:241}
\opn{SHom}_{\leq 1}(\til{A}, B) :=
\bpi_{\leq 1} \bigl( \opn{SHom}(\til{A}, B) \bigr) .
\end{equation}
We refer to $\opn{SHom}_{\leq 1}(\til{A}, B)$ as the {\em Hom groupoid
from $\til{A}$ to $B$}.
Its objects are the DG ring homomorphisms $f : \til{A} \to B$,
and its isomorphisms are the simplicial homotopy classes $[\ga]$ of Keller
homotopies $\ga : f_0 \twoto f_1$.

The proof of Theorem \ref{thm:150}, namely the verification of the Kan condition
in the simplicial set $\opn{SHom}(\til{A}, B)$
when $\til{A}$ is semi-free, relies on Theorem \ref{thm:245} below,
which may be of independent interest.
For every simplicial set $X$ we construct a DG ring
$\opn{N}(X, B)$. When $X = \Simp^q$, the $q$-dimensional combinatorial simplex,
this is precisely the cylinder DG ring
$\opn{Cyl}_{q}(B)$.
Given a horn $\bsym{\La}^{q}_{i}$ in $\Simp^q$, the DG ring
$\opn{N}(\bsym{\La}^{q}_{i}, B)$
turns out to represent horns, as the next theorem shows.
The category of simplicial sets is denoted by $\cat{SSet}$.

\begin{thm} \label{thm:245}
Let $A$ and $B$ be DG rings, and let $\bsym{\La}^{q}_{i}$ be a horn in
$\Simp^q$. Then there is a canonical bijection
\[ \opn{Hom}_{\cat{SSet}} \bigl( \bsym{\La}^{q}_{i},
\opn{SHom}(A, B) \bigr) \cong
\opn{Hom}_{\cat{DGRng}} \bigl( A, \opn{N}(\bsym{\La}^{q}_{i}, B)
\bigr) . \]
\end{thm}

This is repeated as Theorem \ref{thm:152} and proved in Section
\ref{sec:dg-ring-horn}.

Here is a sketch of how Theorem \ref{thm:245} is used to prove Theorem
\ref{thm:150}. A horn
$\si : \bsym{\La}^{q}_{i} \to \opn{SHom}_{}(\til{A}, B)$
corresponds, by Theorem \ref{thm:245}, to a DG ring homomorphism
$f : \til{A} \to \opn{N}(\bsym{\La}^{q}_{i}, B)$.
The inclusion
$\bsym{\La}^{q}_{i} \to \Simp^q$
gives rise to a surjective quasi-isomorphism of DG rings
$w : \opn{N}(\Simp^q, B)  \to \opn{N}(\bsym{\La}^{q}_{i}, B)$.
The lifting property of semi-free DG rings, see Theorem \ref{thm:157},
implies that $f$ lifts to a DG ring homomorphism
$f' : \til{A} \to \opn{N}(\Simp^q, B) = \opn{Cyl}_q(B)$.
A standard fact (see Proposition \ref{prop:445})
says that there is a corresponding map of simplicial sets
$\si' : \Simp^q \to \opn{SHom}_{}(\til{A}, B)$.
It is not hard to see that $\si'$ is a filler of $\si$.

Now two theorems about the Hom groupoid
$\opn{SHom}_{\leq 1}(\til{A}, B)$.

\begin{thm} \label{thm:270}
Let $\til{A}$ be a semi-free DG ring, and let $v : \til{B} \to B$ be a
surjective quasi-isomorphism of DG rings. Then the map of groupoids
\[ \opn{SHom}_{\leq 1}(\til{A}, \til{B}) \to
\opn{SHom}_{\leq 1}(\til{A}, B) \]
induced by $v$ is a surjective equivalence.
\end{thm}

This theorem is repeated as Theorem \ref{thm:302} in Section
\ref{sec:props-hom-grpd}.

A {\em semi-free resolution} of a DG ring $A$ is a surjective quasi-isomorphism
$u : \til{A} \to A$ from a semi-free DG ring $\til{A}$.
Given a DG ring homomorphism $f : A \to B$, and semi-free resolutions
$u : \til{A} \to A$ and $v : \til{B} \to B$, a {\em lifting} of $f$ (with
respect to $u$ and $v$) is a DG ring homomorphism
$\til{f} : \til{A} \to \til{B}$ such that the diagram
\begin{equation} \label{eqn:270}
\begin{tikzcd} [column sep = 8ex, row sep = 5ex]
\til{A}
\ar[r, "{\til{f}}"]
\ar[d, "{u}" swap]
&
\til{B}
\ar[d, "{v}"]
\\
A
\ar[r, "{f}"]
&
B
\end{tikzcd}
\end{equation}
in $\cat{DGRng}$ is commutative.
It is known (see Theorems \ref{thm:180} and \ref{thm:157}) that
semi-free resolutions and liftings always exist.

\begin{thm} \label{thm:271}
Let $f : A \to B$ be a DG ring homomorphism, let
$u : \til{A} \to A$ and $v : \til{B} \to B$
be semi-free resolutions, and let
$\{ \til{f}_i \}_{i \in I}$ be the collection of all liftings
$\til{f}_i : \til{A} \to \til{B}$ of $f$ with respect to $u$ and $v$.
Then\tup{:}
\begin{enumerate}
\item For every $i, j \in I$ there is a distinguished isomorphism
$[\til{g}_{i, j}] : \til{f}_i \to \til{f}_j$ in the groupoid
$\opn{SHom}_{\leq 1}(\til{A}, \til{B})$.

\item The collection of distinguished isomorphisms
$\{ [\til{g}_{i, j}] \}_{i, j \in I}$
satisfies
$[\til{g}_{j, k}] \cd [\til{g}_{i, j}] = [\til{g}_{i, k}]$
for all $i, j, k \in I$, and  $[\til{g}_{i, i}] = \opn{id}_{f_i}$.
\end{enumerate}
\end{thm}

This theorem is copied as Theorem \ref{thm:303} and proved in Section
\ref{sec:props-hom-grpd}.
Item (2) of the theorem says that the collection of distinguished isomorphisms
$\{ [\til{g}_{i, j}] \}_{i, j \in I}$ is {\em contractible}.
See Remark \ref{rem:268} regarding the expected use of this theorem. 

Here are a few remarks on this work, including applications and extensions.

\begin{rem} \label{rem:470}
In the previous version of this paper, dated 2 August 2026, we had another 
theorem in the Introduction, labeled Theorem 0.9, which was repeated in the 
body of the paper as Theorem 7.20, and proved in Section 7. 

That theorem claimed that  given DG rings 
$\til{A}$ and $B$, with $\til{A}$ semi-free, and given a homomorphism 
$f : \til{A} \to B$, the automorphism group 
$\opn{Aut}(f)$ of $f$, as an object of the Hom groupoid 
$\opn{SHom}_{\leq 1}(\til{A}, B)$, is abelian. 
More precisely, the theorem claimed to produce a group isomorphism 
$\bar{\Psi} : \opn{Aut}(f) \iso \opn{H}^{-1}(\opn{Der}_{\til{A}, f}(B))$.
Here $\opn{Der}_{\til{A}, f}(B)$ is the DG abelian group of derivations 
$\til{A} \to B$, where $B$ is a DG $\til{A}$-bimodule via $f$. 

Recently, while preparing a seminar talk on this paper, we had a few sessions 
with Claude (model Opus 5.5 by Anthropic), to clarify some points in the paper. 
These sessions took place during the week of 25 to 30 October 2026. 
Almost all results in the previous version of paper seem to be correct 
according to Claude. 

However, Claude noticed one error, in the proof of the old Theorem 7.20.
Here is a brief discussion of the error, using the setting and notation of 
Section \ref{sec:algeb-high-homs} below. 
 
Steps 1-4 of the proof were correct. Thus there is a {\em 
canonical bijection of pointed sets} $\bar{\Psi}$ as above. 
However, step 5 of the proof had an 
error in it. Claude observed that the element $\si$ appearing there belongs to 
$\opn{Hom}_{\Z}(\til{A}, B)^{-2}$, namely it is a degree $-2$ homomorphism of 
graded abelian groups 
$\si : \til{A} \to B$, 
{\em but it might not be a derivation}.  

Now consider elements 
$g, g', g'' \in \opn{SHom}_{1}(\til{A}, B)$,
whose simplicial homotopy classes satisfy
$[g''] = [g] \cd [g']$ in $\opn{Aut}(f)$. 
Let $\ga := \Psi(g)$, $\ga' := \Psi(g')$, $\ga'' := \Psi(g'')$
be the images in $\opn{Der}_{\til{A}, f}(B)^{-1}$. 
These satisfy 
$\d_{\opn{Hom}}(\si) = \ga - \ga'' + \ga' \in  
\opn{Der}_{\til{A}, f}(B)^{-1}$.
When $\si$ is {\em not a derivation}, the element 
$\d_{\opn{Hom}}(\si)$ might fail to be a coboundary in the DG abelian group 
$\opn{Der}_{\til{A}, f}(B)^{}$. (It is a coboundary in the bigger DG abelian 
group $\opn{Hom}_{\Z}(\til{A}, B)^{}$.)
Therefore the element
\[ [\d_{\opn{Hom}}(\si)] = [\ga] - [\ga''] + [\ga'] \]
of $\opn{H}^{-1}(\opn{Der}_{\til{A}, f}(B))$
might be nonzero. 
This means that the bijection $\bar{\Psi}$ might fail to be a group isomorphism.

In fact there are counterexamples, found by Claude, in which the group  
$\opn{Aut}(f)$ is not abelian.  

Claude also suggested the correct group structure on $\opn{Aut}(f)$. 
For classes $[\ga]$ and $[\ga']$ as above, the group structure transported by 
$\bar{\Psi}$ is 
$[\ga] + [\ga'] + [\ga] \bullet [\ga']$, where $\bullet$ is a bilinear function 
related to the cup product in Hochschild cohomology, and it is often 
nonabelian.  

Since it was false, we decided to delete Theorem 0.9. 
This is the reason for preparing the new version of the paper. 
The deletion does not harm the larger project that is explained in 
Remark 0.17 below. 
Also note that some automorphism groups $\opn{Aut}(f)$ are indeed abelian; this 
is true when $B = \til{A}$ and for $f = \opn{id}_{\til{A}}$, 
due to the Eckmann-Hilton argument. 
\end{rem}

\begin{rem} \label{rem:248}
This is an outline of the extension of our constructions from DG rings to DG
categories. Consider the integral cylinder DG ring
$\opn{Cyl}_{q}(\Z)$.
Given a DG category $\cat{B}$, define the DG category
\begin{equation} \label{eqn:251}
\opn{Cyl}_{q}(\cat{B}) := \cat{B} \ot_{\Z} \opn{Cyl}_{q}(\Z) .
\end{equation}
As $q$ changes, we obtain the simplicial DG category
(in the sense that it is a simplicial object in $\cat{DGCat}$)
\begin{equation} \label{eqn:250}
\opn{Cyl}_{}(\cat{B}) = \{ \opn{Cyl}_{q}(\cat{B}) \}_{q \in \N} .
\end{equation}

Now let $\cat{A}$ be some other DG category. Looking at DG functors
$\cat{A} \to \opn{Cyl}_{q}(\cat{B})$, we get a simplicial set
\begin{equation} \label{eqn:252}
\opn{Hom}_{\cat{DGCat}}(\cat{A}, \opn{Cyl}(\cat{B})) =
\bigl\{ \opn{Hom}_{\cat{DGCat}}(\cat{A}, \opn{Cyl}_{q}(\cat{B}))
\bigr\}_{q \in \N} .
\end{equation}
This is the generalization of (\ref{eqn:285}) and (\ref{eqn:286}). It is
reasonable to ask whether the generalization of Theorem \ref{thm:150} holds,
namely that when $\cat{A}$ is a semi-free DG category, the simplicial set
(\ref{eqn:252}) is a Kan complex. Presumably the answer is positive; and a proof
would be a suitable modification of our proof of Theorem \ref{thm:150}, using
the DG categories
$\opn{N}(\bsym{\La}^{q}_{i}, \cat{B}) := \cat{B} \ot_{\Z}
\opn{N}(\bsym{\La}^{q}_{i}, \Z)$,
see Definition \ref{dfn:195}.
\end{rem}

\begin{rem} \label{rem:270}
In this paper -- with the exception of Section \ref{sec:prelims-dg} -- we
consider DG rings in the absolute setting.
It is not hard to extend the whole discussion to a relative setting, namely to
fix a base commutative DG ring $\K$, and to replace $\cat{DGRng}$ with
the category $\cat{DGRng} \centover \K$ of central DG $\K$-rings.
Likewise, in Remark \ref{rem:248} we may consider $\K$-linear DG categories.
Again, for the sake of streamlining the discussion, we left
this generalization out.
\end{rem}

\begin{rem} \label{rem:266}
A {\em simplicial enrichment} of $\cat{DGRng}$, or more generally of
$\cat{DGCat}$, is one incarnation of its $\infty$-derived category.
The known constructions of such simplicial enrichments are indirect, and rely
on difficult arguments involving Quillen model structures; see
\cite{Ta} and its references.
We were hoping to make use of Theorem \ref{thm:150}
to obtain an explicit simplicial enrichment of $\cat{DGRng}$; but
we were not able to describe the horizontal composition of higher morphisms. So
we leave this as a question: Let $\til{A}$, $\til{B}$ and $\til{C}$ be
semi-free DG rings. Is it
possible to define a natural map of simplicial sets
\begin{equation} \label{eqn:247}
\opn{SHom}(\til{A}, \til{B}) \times
\opn{SHom}(\til{B}, \til{C}) \to
\opn{SHom}(\til{A}, \til{C})
\end{equation}
that would be the horizontal composition of a simplicial enrichment of
$\cat{DGRng}$~?
\end{rem}

\begin{rem} \label{rem:268}
The best we can do at present is a much weaker than a simplicial enrichment of
$\cat{DGRng}$, yet it is still important in our opinion, and as far as we can 
tell it is new. 

We expect to have an {\em explicit construction} of a strict $(2, 1)$-category, 
that we call the {\em $(2, 1)$-derived category of DG rings}, denoted by
$\cat{D}_{(2, 1)}(\cat{DGRng})$. Here is an outline of the construction. 

Let $\til{A}$, $\til{B}$ and $\til{C}$ be semi-free DG rings. 
In our investigations we found that {\em there are two canonical formulas 
for horizontal composition} between the  Hom groupoids 
\begin{equation} \label{eqn:415}
\opn{SHom}_{\leq 1}(\til{A}, \til{B}) \times
\opn{SHom}_{\leq 1}(\til{B}, \til{C}) \to
\opn{SHom}_{\leq 1}(\til{A}, \til{C}) .
\end{equation}
Both formulas arise from composing Keller homotopies and DG ring homomorphisms.

To have a strict $2$-category structure, these two horizontal compositions must 
be {\em equal}. This equality gives a single canonical horizontal composition, 
and it also implies the {\em interchange property}. 

Our hope was that having abelian automorphism groups would make proving the 
equality of compositions easier.  But with the failure of the old Theorem 0.9
(as explained in Remark \ref{rem:470} above), 
things became more challenging. 

Thus we asked Claude if it can prove directly that the two canonical horizontal 
compositions in (\ref{eqn:415}) are equal. 
The response was a nice proof, which is also conceptual: the equality of the 
two horizontal compositions is evidenced by a particular $2$-simplex in the Kan 
complex $\opn{SHom}_{}(\til{A}, \til{C})$,
which depends canonically on the input. 

After that, we asked Claude if it can prove that this horizontal composition 
gives rise 
to a strict $(2, 1)$-category structure. The response was a complete proof of 
this assertion. 

We went over the proofs, and they appear to be sound, but we need more time to 
verify the details. If all is well, it means that our prediction here is 
correct.

Theorem \ref{thm:271} is expected to produce a pseudofunctor 
$\cat{DGRng} \to \cat{D}_{(2, 1)}(\cat{DGRng})$,
sending a DG ring $A$ to a semi-free resolution $\til{A}$ of it, 
and sending a homomorphism $f : A \to B$ to a lifting 
$\til{f} : \til{A} \to \til{B}$. 
We will then attempt to show that the 
strict $(2, 1)$-category 
$\cat{D}_{(2, 1)}(\cat{DGRng})$
that we construct explicitly is the truncation $\leq 2$ of the 
$\infty$-derived category of 
$\cat{DGRng}$ from Remark \ref{rem:266}. 
We intend to write this in the future paper \cite{Ye3}. 
\end{rem}

Here are remarks on related work by other authors. In these remarks we are
going to use the notation of the present paper, to make the comparison legible.

\begin{rem} \label{rem:245}
In the paper \cite{Ta}, Tabuada constructs a simplicial enrichment of the
category $\cat{DGCat}$ of (small) DG categories. The starting point is the
formation of the {\em path object}
$\opn{P}(\cat{B})$ associated to a DG category $\cat{B}$. It appears that
$\opn{P}(\cat{B})$ is the same as the DG category
$\opn{Cyl}_{1}(\cat{B})$ from equation (\ref{eqn:251}) above.
Tabuada then proceeds to prove that $\opn{P}(\cat{B})$ plays the role of a path
object for the Quillen model structure on $\cat{DGCat}$ that was constructed in
their earlier paper.

At this point our work diverges from the work of Tabuada.
They do not have higher cylinders (or path objects) for $q > 1$. Instead, they
prove that there is a zig-zag of Quillen adjunctions
between $\cat{DGCat}$ and the category
$\cat{SSetCat}$ of simplicially enriched categories. The composed functor
$\cat{DGCat} \to \cat{SSetCat}$ is their simplicial enrichment of
$\cat{DGCat}$.
\end{rem}

\begin{rem} \label{rem:250}
Faonte, in the paper \cite{Fa}, produces a simplicial enrichment of
$\cat{DGCat}$ in the following way. To a pair of DG categories
$\cat{A}$ and $\cat{B}$ they consider the DG category
$\cat{Fun}_{\opn{A}_{\infty}}(\cat{A}, \cat{B})$
of $\opn{A}_{\infty}$-functors $F : \cat{A} \to \cat{B}$. They then take
the DG nerve of
$\cat{Fun}_{\opn{A}_{\infty}}(\cat{A}, \cat{B})$,
which is a simplicial set, and prove that it is weakly homotopy equivalent
to the simplicial enrichment of Tabuada, as in Remark \ref{rem:245}.
We do not know if there is any further relationship between our work and the
work of Faonte.
\end{rem}

\begin{rem} \label{rem:246}
Holstein, in the paper \cite{Ho}, studies simplicial resolutions of DG
categories. For a DG category $\cat{B}$ they produce a simplicial DG category
$\cat{B}_{\bullet} = \{ \cat{B}_q \}_{q \in \N}$. The constructions are rather
involved; but in Example 3.5 they show that their $\cat{B}_1$ agrees with
Tabuada's path object,
see Remark \ref{rem:245} above, and hence also with our
$\opn{Cyl}_{1}(\cat{B})$. This might indicate that Holstein's simplicial
DG category $\cat{B}_{\bullet}$ agrees with our simplicial DG category
$\opn{Cyl}_{}(\cat{B})$ from Remark \ref{rem:248}. However we
could not verify the details.
\end{rem}

We end the Introduction with an outline of the paper.

\medskip \noindent
{\bf Section 1.} We recall Keller derivations and the Keller cylinder DG ring.
Then we recall the definition of semi-free DG rings, and prove some results
about their lifting properties, extending results from \cite{Ye1} and
\cite{Ye2}.

\medskip \noindent
{\bf Section 2.}
This section is about simplicial sets. We describe the horn
$\bsym{\La}^{q}_{i}$ as a colimit, in the category
$\cat{SSet}$ of simplicial sets, of a finite direct system
$\{ \Simp^{q_j} \}_{j \in J}$ of standard simplices.
Then we do the same for the semi-simplicial set
$(\bsym{\La}^{q}_{i})^{\mrm{nd}}$ of nondegenerate simplices in
$\bsym{\La}^{q}_{i}$.

\medskip \noindent
{\bf Section 3.}
Here we introduce the DG ring  $\opn{N}(X, B)$ associated functorially to a
simplicial set $X$ and a DG ring $B$. For the simplicial set $\Simp^q$
we get the cylinder DG ring
$\opn{Cyl}_{q}(B) := \opn{N}(\Simp^q, B)$.
The cosimplicial simplicial set structure on
$\{ \Simp^q \}_{q \in \N}$
gives rise to the simplicial DG ring structure on
$\opn{Cyl}_{}(B) = \{ \opn{Cyl}_{q}(B) \}_{q \in \N}$.

\medskip \noindent
{\bf Section 4.}
In this section we prove the key technical result of the paper, Theorem
\ref{thm:152} (a repetition of Theorem \ref{thm:245}), which asserts that
the DG ring
$\opn{N}(\bsym{\La}^{q}_{i}, B)$
encodes horns in the simplicial set $\opn{SHom}_{}(A, B)$.

\medskip \noindent
{\bf Section 5.}
Here we prove the main result of the paper, namely Theorem \ref{thm:150},
asserting that  $\opn{SHom}_{}(\til{A}, B)$ is a Kan complex when $\til{A}$ is
semi-free. We then introduce the Hom groupoid
$\opn{SHom}_{\leq 1}(\til{A}, B)$,
and describe it from the simplicial side.

\medskip \noindent
{\bf Section 6.}
This section focuses on the algebraic structure of
$\opn{SHom}_{2}(A, B)$; or, in other words, what a DG ring homomorphism
$h : A \to \opn{Cyl}_{2}(B)$ looks like.
Theorem \ref{thm:390} says that $h$ consists of a few DG ring homomorphisms and
derivations $A \to B$, satisfying compatibility conditions.

\medskip \noindent
{\bf Section 7.}
Here we study the algebraic structure of the Hom groupoid
$\opn{SHom}_{\leq 1}(\til{A}, B)$
when $\til{A}$ is semi-free. The results proved in this section are Theorems
\ref{thm:270} and \ref{thm:271} that were mentioned above.

\section{Preliminaries About DG Rings} \label{sec:prelims-dg}

In this section we recall some material on DG rings, from the paper \cite{Ye1}
and the book \cite{Ye2}.
We also prove upgraded versions of a few theorems from these references.

The setting here is more general than is required for the
current paper, regarding the base. In this section we have a fixed base CDG
ring $A$, and we study central DG $A$-rings.
In the rest of the paper we work in the absolute setting, i.e.\ $A = \Z$.
The reason is that the theorems in this section will be needed elsewhere in
their full generality.

A {\em DG ring} is a graded ring $A = \bigoplus_{i \in \Z} A^i$,
with a differential $\d$ of degree $1$ satisfying
$\d \circ \d = 0$, and the graded Leibniz rule
$\d(a \cd b) = \d(a) \cd b + (-1)^i \cd a \cd \d(b)$ for
$a \in A^i$ and $b \in A^j$. Traditionally, $A$ is called a unital associative
cochain DG algebra. Let $\cat{DGRng}$ denote the category of DG rings and DG
ring homomorphisms.

Suppose $A$ is a fixed DG ring. A DG $A$-ring is a DG ring $B$ equipped
with a DG ring homomorphism $u_B : A \to B$, called the structural homomorphism.
A homomorphism of DG $A$-rings $f : B \to C$ is a DG ring homomorphism $f$ such
that $f \circ u_B = u_C$. The  category of DG rings  is denoted by
$\cat{DGRng} \over A$.

Let $A$ be some DG ring. A homogeneous element $a \in A^i$ is called central if
for every $b \in A^j$ there is equality
$b \cd a = (-1)^{i \cd j} \cd a \cd b$.
The {\em center} of $A$ consists of elements that are sums of homogeneous
central elements, and it is denoted by $\opn{Cent}(A)$. It is easy to check that
$\opn{Cent}(A)$ is a DG subring of $A$.

A DG ring $A$ is called {\em commutative} if it is nonpositive (i.e.\
$A = \bigoplus_{i \leq 0} A^i$) and strongly commutative
(i.e.\ $b \cd a = (-1)^{i \cd j} \cd a \cd b$ for all
$a \in A^i$ and $b \in A^j$, and also $a \cd a = 0$ if $i$ is odd).
The abbreviation is {\em CDG ring}.
Clearly $\opn{Cent}(A) = A$ for a CDG ring $A$.

Now fix some CDG ring $A$. A DG $A$-ring $B$ is called {\em central} if
the structural homomorphism $u_B : A \to B$ is a central homomorphism,
i.e.\ $u_B(A) \sub \opn{Cent}(B)$.
The category of central DG $A$-rings is a full subcategory of
$\cat{DGRng} \over A$, and it is denoted by $\cat{DGRng} \centover A$.
For $A = \Z$ we have
$\cat{DGRng} \centover \Z = \cat{DGRng} \over \Z = \cat{DGRng}$.

\begin{conv} \label{conv:330}
From here to the end of the section we have a fixed nonzero
CDG ring $A$. By default, all DG rings are central DG $A$-rings, and all
homomorphisms between them are over $A$.
\end{conv}

The following definitions and proposition first appeared in the
papers \cite{Ke1} and \cite{Ke2} of Keller. See also
\cite[Sections 1 and 4]{Ye1}.

\begin{dfn} \label{dfn:330}
Let $f_0, f_1 : B \to C$ be DG $A$-ring homomorphisms.
A {\em Keller homotopy} $\ga : f_0 \twoto f_1$ is an $A$-linear homomorphism
$\ga : B \to C$ of degree $-1$, satisfying these two conditions:
\begin{itemize}
\rmitem{i} Twisted derivation: For every $b_0 \in B^{i_0}$ and $b_1 \in
B^{i_1}$ there is equality
\[ \ga(b_0 \cd b_1) = \ga(b_0) \cd f_1(b_1) +
(-1)^{i_0} \cd f_0(b_0) \cd \ga(b_1) . \]

\rmitem{ii} Homotopy:
$\d_C \circ \ga + \ga \circ \d_B = f_1 - f_0$.
\end{itemize}
\end{dfn}

\begin{dfn} \label{dfn:331}
The  {\em integral Keller cylinder DG ring} is the matrix DG ring
\[ \opn{Cyl}_{\mrm{Ke}}(\Z) :=
\bmat{\Z & \Z[-1] \\ 0 & \Z} . \]
The differential is the commutator with the degree $1$ element
$\sbmat{0 & \bmsp 1[-1] \\ 0 & 0}$.
\end{dfn}

To make things more concrete, let $y$ be a variable of degree $1$, so that
there is an isomorphism of DG abelian groups (with zero differentials)
$\Z[-1] \iso y \ot \Z$, $1[-1] \mapsto y \ot 1$.
Then we can write
\begin{equation} \label{eqn:235}
\opn{Cyl}_{\mrm{Ke}}(\Z) =
\bmat{1 \ot \Z & y \ot \Z \\ 0 & 1 \ot \Z} .
\end{equation}
The elements
$e_0 := \sbmat{1 & 0 \\ 0 & 0}$ and
$e_1 := \sbmat{0 & 0 \\ 0 & 1}$
have degree $0$, and the element
$e_{0, 1} := \sbmat{0 & y \\ 0 & 0}$
has degree $1$. The multiplication in
$\opn{Cyl}_{\mrm{Ke}}(\Z)$ is the usual matrix multiplication, and
the differential is
$\d(e_0) = -e_{0, 1}$, $\d(e_1) = e_{0, 1}$, and
$\d(e_{0, 1}) = 0$.

\begin{dfn} \label{dfn:332}
For an arbitrary DG $A$-ring $B$ we define its Keller cylinder DG ring to be
\[ \opn{Cyl}_{\mrm{Ke}}(B) :=
\opn{Cyl}_{\mrm{Ke}}(\Z) \ot_{\Z} B =
\bmat{1 \ot B & y \ot B \\ 0 & 1 \ot B} . \]
\end{dfn}

\begin{prop} \label{prop:330}
Let $f_0, f_1 : B \to C$ be DG $A$-ring homomorphisms, and let
$\ga : B \to C$ be  $A$-linear homomorphism of degree $-1$.
The two conditions below are equivalent:
\begin{itemize}
\rmitem{i} The homomorphism $\ga$ is a Keller homotopy $f_0 \twoto f_1$.

\rmitem{ii} The $A$-linear homomorphism
$g : B \to \opn{Cyl}_{\mrm{Ke}}(C)$,
$g := \sbmat{f_0 & \bmsp y \cd \ga \\ 0 & f_1}$, is a DG ring
homomorphism.
\end{itemize}
\end{prop}

This is is a straightforward calculation. See
\cite[Proposition 4.5]{Ye1}, following \cite[Theorem 4.3(c)]{Ke2}.

According to Keller (private communication), when the DG $A$-ring $B$ is
noncommutative semi-free (see Definition \ref{dfn:180} below),
the Keller homotopies
form an equivalence relation on the set
$\opn{Hom}_{\cat{DGRng}}(B, C)$.
Our Theorem \ref{thm:150} says more: this equivalence relation is by
isomorphisms in a canonical groupoid structure on this set of objects.

In the remainder of this section we provide upgraded versions of theorems from
\cite{Ye1} and \cite{Ye2}.

A filtered graded set $(X, F)$ is a graded set
$X = \coprod_{i \in \Z} X^i$, and
an ascending filtration $\{ F_j(X) \}_{j \geq -1}$
of $X$ by graded subsets,
such that $F_{-1}(X) = \varnothing$ and
$X = \bigcup_j F_j(X)$.

Suppose we are given a filtered graded set $(X, F)$.
Let $A^{\natural}$ the graded ring gotten by forgetting the differential of $A$.
Define $A^{\natural} \bra{X} := A^{\natural} \otimes_{\Z} \Z \bra{X}$
to be the noncommutative graded polynomial ring in the set of variables $X$
over $A^{\natural}$.
The graded ring $A^{\natural} \bra{X}$ is filtered by
$F_j(A^{\natural} \bra{X}) := A^{\natural} \bra{F_j(X)}$
for $j \geq 0$, and $F_{-1}(A \bra{X}) := 0$.

The next couple of definitions are
\cite[Definitions 12.8.3 and 12.8.6]{Ye2}, generalized in two ways:
Changing from a base ring $\K$ to a base CDG ring $A$, and
removing the requirement that the DG rings are nonpositive.
Also we require a resolution $u : \til{B} \to B$ to be a {\em surjective}
quasi-isomorphism, as in \cite[Definitions 3.8 and 3.15]{Ye1}.

\begin{dfn} \label{dfn:180}
A {\em noncommutative semi-free DG $A$-ring} is a central DG $A$-ring
$\til{B}$, which admits an isomorphism of graded $A^{\natural}$-rings
$\til{B}^{\natural} \cong A^{\natural} \bra{X}$
for some filtered graded set $(X, F)$, such that under this isomorphism we have
$\d_{\til{B}}(F_j(X)) \sub F_{j - 1}(A \bra{X})$
for all $j \geq 0$.
Such a filtered graded set $(X, F)$ is called a DG $A$-ring semi-basis of
$\til{B}$.
\end{dfn}

\begin{dfn} \label{dfn:181}
Let $B$ be a central DG $A$-ring. A {\em noncommutative semi-free DG $A$-ring
resolution of $B$} is a surjective quasi-isomorphism
$u : \til{B} \to B$ of DG $A$-rings, where $\til{B}$ is a
noncommutative semi-free DG $A$-ring.
\end{dfn}

The next theorem is an improvement of \cite[Theorem 12.8.7]{Ye2}.
The change is that here the resolution $u$ is a {\em surjective}
quasi-isomorphism. Cf.\ \cite[Exercise 12.8.19]{Ye2}.

\begin{thm}[Existence of Resolutions] \label{thm:180}
Let $A$ be a nonzero CDG ring, and let $B$ be a central DG $A$-ring. Then there
exists a noncommutative semi-free DG $A$-ring resolution
$u : \til{B} \to B$.
\end{thm}

\begin{proof}
We only make the necessary adjustments to the proof of
\cite[Theorem 12.8.7]{Ye2}, to make the homomorphism $u$ surjective. These
adjustments are mostly is step 1 of the proof. The strategy of the proof, as it
is stated there, remains unchanged.

We are going to use \cite[Lemmas 12.8.8 and 12.8.10]{Ye2}. These lemmas are
stated for a base ring $\K$, but the proofs work just as well for a base CDG
ring $A$. Compare to  \cite[Lemmas 3.20 and 3.19]{Ye1}, where the base is a CDG
ring $A$, yet the DG rings are nonpositive.

\medskip \noindent
Step 1.
In this step we treat the cases $j = 0, 1$. (Unlike the proof in \cite{Ye2}, in
which step 1 treated only $j = 0$).

Choose a graded set $Y_1$ and a collection
$\{ b_y \}_{y \in Y_1}$ of elements
$b_y \in B^{\opn{deg}(y)}$, such that this collection generates
$B^{\natural}$ as an $A^{\natural}$-ring.
Partition $Y_1$ as follows:
$Y_1' := \{ y \in Y_1 \mid \d_B(b_y) \neq 0 \}$
and
$Y_1'' := \{ y \in Y_1 \mid \d_B(b_y) = 0 \}$.

Let $Y'_0$ be a graded set equipped with a degree $1$ bijection
$\ze : Y'_1 \to Y'_0$. Define the collection
$\{ b_y \}_{y \in Y'_0}$ of elements
$b_y \in B^{\opn{deg}(y)}$ by the formula
$b_{\ze(z)} := \d_B(b_z)$ for $z \in Y'_1$.
Note that $b_y$ is a cocycle for every $y \in Y_0'$.

Next let $Y''_0$ be a graded set, with a collection
$\{ b_y \}_{y \in Y''_0}$ of cocycles
$b_y \in \opn{Z}(B)^{\opn{deg}(y)}$,
such that this collection generates $\opn{Z}(B)$ as a $\opn{Z}(A)$-ring.

Define the graded sets $Y_0 := Y_0' \coprod Y_0''$ and
$Y_1 := Y_1' \coprod Y_1''$. Then define the graded sets
$F_0(X) := Y_0$ and $F_1(X) := Y_0 \coprod Y_1$.

Let $F_0(\til{B}) := A \bra{F_0(X)}$, with the unique differential $\d$
extending $\d_A$, and such that $\d(y) = 0$ for all $y \in Y_0 = F_0(X)$.
This is possible by \cite[Lemma 12.8.8]{Ye2}.
Let $F_0(u) : F_0(\til{B}) \to B$ be the unique DG $A$-ring homomorphism such
that  $F_0(u)(y) = b_y$ for all $y \in Y_0$.
This can be done by \cite[Lemma 12.8.10]{Ye2}.
This homomorphism has the property that
$F_0(u)(\opn{Z}(F_0(\til{B}))) = \opn{Z}(B)$,
and therefore
$\opn{H}(F_0(u)) : \opn{H}(F_0(\til{B})) \to \opn{H}(B)$ is surjective.

Finally let $F_1(\til{B}) := A \bra{F_1(X)}$, with the unique differential $\d$
extending $\d_A$, and such that $\d(y) = 0$ for all $y \in Y_0 = F_0(X)$,
$\d(y) = 0$ for all $y \in Y_1''$,
and $\d(y) = \ze(y)$ for $y \in Y_1'$.
Again, This is possible by \cite[Lemma 12.8.8]{Ye2}.
Let $F_1(u) : F_1(\til{B}) \to B$ be the unique DG $A$-ring homomorphism such
that $F_1(u)(y) := b_y$ for all $y \in F_1(X) = Y_0 \coprod Y_1$.
Again, This is possible by \cite[Lemma 12.8.10]{Ye2}.
The homomorphism $F_1(u) : F_1(\til{B}) \to B$ is surjective.
Of course $F_1(u)|_{F_0(\til{B})} = F_0(u)$.
This implies that
$\opn{H}(F_1(u)) : \opn{H}(F_1(\til{B})) \to \opn{H}(B)$ is surjective.

\medskip \noindent
Steps 2 and 3. Here the induction is on $j \geq 1$ (instead of $j \geq 0$
in the proof of \cite[Theorem 12.8.7]{Ye2}),
and the inductive hypothesis includes the condition that
$F_j(u) : F_j(\til{B}) \to B$ is surjective
(not just that $\opn{H}(F_j(u))$ is surjective).

\medskip \noindent
Step 4. No changes, same as in the proof of \cite[Theorem 12.8.7]{Ye2}.
\end{proof}

The following theorem is a generalization of the noncommutative part of
\cite[Theorem 3.22]{Ye1}. The change is that the DG rings are not required to
be nonpositive in Theorem \ref{thm:157}. We also correct a small error
in the proof of \cite[Theorem 3.22]{Ye1}, see Remark \ref{rem:185}.

\begin{thm}[Existence of Liftings] \label{thm:157}
Let $A$ be a nonzero CDG ring, let $\til{B}$ be a noncommutative semi-free
central DG $A$-ring, let $v : \til{C} \to C$ be a surjective
quasi-isomorphism of central DG $A$-rings, and let
$u : \til{B} \to C$ be a DG $A$-ring homomorphism. Then there is a lifting
$\til{u} : \til{B} \to \til{C}$ of $u$, namely a DG $A$-ring homomorphism
$\til{u}$ such that $u = v \circ \til{u}$.
\end{thm}

Here is the diagram in $\cat{DGRng} \centover A$ depicting the theorem.

\[ \begin{tikzcd} [column sep = 10ex, row sep = 4ex]
&
\til{C}
\ar[d, "{v}"]
\\
\til{B}
\ar[r, "{u}" swap]
\ar[ur, dashed, "{\til{u}}"]
&
C
\end{tikzcd} \]

We need a lemma first.

\begin{lem} \label{lem:185}
Let $\phi : \til{M} \to M$ be surjective quasi-isomorphism of DG modules.
\begin{enumerate}
\item Given an element $m \in M^i$ such that $\d(m) = 0$, there exists an
element $\til{m} \in \til{M}^i$ such that
$\d(\til{m}) = 0$ and $\phi(\til{m}) = m$.

\item Given an element $\til{m} \in \til{M}^i$ such that $\d(\til{m}) = 0$
and $\phi(\til{m}) = 0$, there exists an element
$\til{m}' \in \til{M}^{i - 1}$ such
that $\d(\til{m}') = \til{m}$ and $\phi(\til{m}') = 0$.
\end{enumerate}
\end{lem}

\begin{proof}
(1) For this item, the only condition on $\opn{H}(\phi)$ that is needed is
surjectivity. Since $\phi$ is surjective, it follows that
$\opn{B}(\phi) : \opn{B}(\til{M}) \to \opn{B}(M)$ is surjective.
Consider the commutative diagram
\[ \begin{tikzcd}
0
\ar[r]
&
\opn{B}(\til{M})
\ar[r]
\ar[d, "{\opn{B}(\phi)}"]
&
\opn{Z}(\til{M})
\ar[r]
\ar[d, "{\opn{Z}(\phi)}"]
&
\opn{H}(\til{M})
\ar[r]
\ar[d, "{\opn{H}(\phi)}"]
&
0
\\
0
\ar[r]
&
\opn{B}(M)
\ar[r]
&
\opn{Z}(M)
\ar[r]
&
\opn{H}(M)
\ar[r]
&
0
\end{tikzcd} \]
of graded modules, with exact rows, and such that $\opn{B}(\phi)$ and
$\opn{H}(\phi)$ are surjective. Then the middle homomorphism  $\opn{Z}(\phi)$
is also surjective.
The element $m$ belongs to $\opn{Z}^i(M)$, and therefore there is an element
$\til{m}$ in $\opn{Z}^i(\til{M})$ such that $\phi(\til{m}) = m$.

\medskip \noindent
(2) Define $\til{K} := \opn{Ker}(\phi)$. Because $\phi$ is a
surjective quasi-isomorphism, the DG module $\til{K}$ is acyclic.
The element $\til{m}$ belongs to $\opn{Z}^i(\til{K})$.
Hence there exists an element $\til{m}' \in \til{K}^{i - 1}$ such that
$\d(\til{m}') = \til{m}$.
\end{proof}

\begin{proof}[Proof of Theorem \tup{\ref{thm:157}}]
We repeat the proof of \cite[Theorem 3.22]{Ye1}, with the needed modifications,
and with a correction of the error.

Let $(X, F)$ be a DG $A$-ring semi-basis of $\til{B}$.
For $j \geq 0$ let
$F_j(\til{B})$ be the $A$-subring of $\til{B}$ generated by $F_j(X)$.
So $F_j(\til{B})$ is a DG $A$-subring of $\til{B}$,
$F_j(\til{B})^{\natural} = A^{\natural} \bra{F_j(X)}$,
and
$\til{B} = \bigcup_j F_j(\til{B})$.

We will construct a consistent sequence of DG $A$-ring homomorphisms
$\til{u}_j : F_j(\til{B}) \to \til{C}$,
satisfying $v \circ \til{u}_j = u$ on $F_j(\til{B})$.
The construction is by recursion on $j \in \N$.
Then $\til{u} := \lim_{j \to} \til{u}_j$ will have the required properties.

We start with $j = 0$. Take any $x \in F_0(X)$, and let
$k := \opn{deg}(x)$. Since $\d(x) = 0$ and $v$ is a surjective
quasi-isomorphism, Lemma \ref{lem:185}(1) says that
there exists an element $\til{c} \in \til{C}^k$ such that
$\d(\til{c}) = 0$ and $v(\til{c}) = u(x)$.
We define $\til{u}_0(x) := \til{c}$.
The resulting function
$\til{u}_0 : F_0(X) \to \til{C}$
extends uniquely to a DG $A$-ring homomorphism
$\til{u}_0 : F_0(\til{B}) \to \til{C}$, by
\cite[Lemma 12.8.10]{Ye2}.
Since $v(\til{u}_0(x)) = u(x)$ for all $x \in F_0(X)$, it follows that
$v \circ \til{u}_0 = u$ on $F_0(\til{B})$.

Next consider any $j \in \N$, and
assume a DG $A$-ring homomorphism
$\til{u}_{j} : F_j(\til{B}) \to \til{C}$ has been constructed, satisfying
$v \circ \til{u}_j = u$ on $F_j(\til{B})$.

Take any element $x \in F_{j + 1}(X) - F_j(X)$,
and let $k := \opn{deg}(x)$.
Since $v$ is surjective, there exists
$\til{c} \in \til{C}^{k}$ such that
$v(\til{c}) = u(x)$.

The element $\d(x)$ belongs to $F_j(\til{B})$, and it is a cocycle in this DG
ring. Therefore the element
$\til{u}_j(\d(x))$ is a cocycle in $\til{C}$.
Define
$\til{c}' := \til{u}_j(\d(x)) - \d(\til{c}) \in \til{C}^{k + 1}$,
which is also a cocycle in $\til{C}$.
We have
\[ v(\til{c}') = v(\til{u}_j(\d(x)) - v(\d(\til{c})) =
u(\d(x)) - v(\d(\til{c})) = \d(u(x) - v(\til{c})) = \d(0) = 0 . \]
By Lemma \ref{lem:185}(2) there exists an element
$\til{c}'' \in \til{C}^k$ such that
$\d(\til{c}'') = \til{c}'$
and $v(\til{c}'') = 0$.
Let us define
$\til{u}_{j + 1}(x) := \til{c} + \til{c}''$.
Then
\begin{equation} \label{eqn:186}
v(\til{u}_{j + 1}(x)) = v(\til{c} + \til{c}'') = v(\til{c}) = u(x)
\end{equation}
\begin{equation} \label{eqn:305}
\d(\til{u}_{j + 1}(x)) = \d(\til{c} + \til{c}'') =
\d(\til{c}) + \til{c}' = \til{u}_j(\d(x)) .
\end{equation}

In this way we obtain a degree $0$ function
$\til{u}_{j + 1} : F_{j + 1}(X) \to \til{C}$,
which restricts to $\til{u}_{j}$ on $F_{j}(X)$.
According to \cite[Lemma 12.8.10]{Ye2}, the function
$\til{u}_{j + 1}$ extends uniquely to a homomorphism of graded
$A^{\natural}$-rings
$\til{u}_{j + 1} : F_{j + 1}(\til{B})^{\natural} \to
\til{C}^{\natural}$.
The homomorphism $\til{u}_{j + 1}$ restricts to $\til{u}_{j}$ on
$F_{j}(\til{B})$.
By equation (\ref{eqn:305}), $\til{u}_{j + 1}$ is a DG ring homomorphism.
Equation (\ref{eqn:186}) implies that
$v \circ \til{u}_{j + 1} = u$ on $F_{j + 1}(\til{B})$.
\end{proof}

Here is a construction in DG rings, which will be used in the proof of
the next theorem, and also in Sections \ref{sec:algeb-high-homs} and
\ref{sec:props-hom-grpd}.

\begin{prop} \label{prop:425}
Suppose  $B = \boplus_{i \geq 0} B^i$ is a nonnegative DG ring, and
$C_0 \to C_1 \to \cdots$ is direct system of DG rings, with transition
homomorphisms $f_{i, j} : C_i \to C_j$.
Define the DG abelian group
\[ D := \boplus_{i \geq 0} B^i \ot_{\Z} C_i . \]
Next define the multiplication
\[ \begin{aligned}
&
\bigl( B^{i_1} \ot_{\Z} C_{i_1}^{j_1} \bigr) \times
\bigl( B^{i_2} \ot_{\Z} C_{i_2}^{j_2} \bigr) \to
\bigl( B^{i_1 + i_2} \ot_{\Z} C_{i_1 + i_2}^{j_1 + j_2} \bigr)
\\ &
(b_1 \ot c_1) \cdot (b_2 \ot c_2) :=
(-1)^{j_1 \cd i_2} \cd (b_1 \cd b_2) \ot
\bigl( f_{i_1, i_1 + i_2}(b_1) \cd f_{i_2, i_1 + i_2}(b_2) \bigr)
\end{aligned} \]
on $D$. Then the DG abelian group $D$ is a DG ring.
\end{prop}

The proof is an easy calculation. The prototypical example is:

\begin{exa} \label{exa:310}
Take $B := \opn{Cyl}_{\mrm{Ke}}(\Z)$, $C$ arbitrary, and $C_i := C$.
Then
\[ \opn{Cyl}_{\mrm{Ke}}(C) = \opn{Cyl}_{\mrm{Ke}}(\Z) \ot_{\Z} C =
\boplus_{0 \leq i \leq 1}
\opn{Cyl}_{\mrm{Ke}}(\Z)^i \ot_{\Z} C_i . \]
This example has a matrix presentation:
\[ \begin{aligned}
&
i = 0
\bmsp \bmsp
\opn{Cyl}_{\mrm{Ke}}(\Z)^0 \ot_{\Z} C_0 =
\bmat{1 \ot C &  0 \\ 0 & 1 \ot C}
\\
&
i = 1
\bmsp \bmsp
\opn{Cyl}_{\mrm{Ke}}(\Z)^1 \ot_{\Z} C_1 =
\bmat{0 & y \ot C \\ 0 & 0} .
\end{aligned} \]
Cf.\ Definition \ref{dfn:332}.
\end{exa}

The next theorem is a generalization of the noncommutative part of
\cite[Theorem 3.22]{Ye1}, removing the nonpositivity condition. The proof is
also much shorter.

\begin{thm}[Existence of Homotopies] \label{thm:312}
Let $A$ be a nonzero CDG ring, let $\til{B}$ be a noncommutative semi-free
DG $A$-ring, let $v : \til{C} \to C$ be a surjective
quasi-isomorphism of central DG $A$-rings, let $u : \til{B} \to C$ be a DG
$A$-ring homomorphism, and let
$\til{u}_0, \til{u}_1 : \til{B} \to \til{C}$
be liftings of $u$. Then there exists a Keller homotopy
$\ga : \til{u}_0 \twoto \til{u}_1$, such that $v \circ \ga = 0$.
\end{thm}

See these diagrams for an illustration.

\[ \begin{tikzcd} [column sep = 10ex, row sep = 4ex]
&
\til{C}
\ar[d, "{v}"]
\\
\til{B}
\ar[r, "{u}" swap]
\ar[ur, "{\til{u}_i}"]
&
C
\end{tikzcd}
\qquad
\begin{tikzcd} [column sep = 14ex, row sep = 8ex]
\til{B}
\arrow[r, bend left, "\til{u}_0", "" {name = U, inner sep = 1pt, below}]
\arrow[r, bend right, "\til{u}_1" {below}, "" {name=D , inner sep = 1pt}]
&
\til{C}
\arrow[Rightarrow, from = U, to = D, "\ga"]
\end{tikzcd} \]

\begin{proof}
Define the DG ring $D$ to be
\begin{equation} \label{eqn:330}
D :=
\bigl( \opn{Cyl}_{\mrm{Ke}}(\Z)^0 \ot_{\Z} \til{C} \bigr)
\oplus
\bigl( \opn{Cyl}_{\mrm{Ke}}(\Z)^1 \ot_{\Z} C \bigr) =
\bmat{1 \ot \til{C} & y \ot C \\ 0 & 1 \ot \til{C}} .
\end{equation}
This is the construction in Proposition \ref{prop:425},
with first tensor factor $\opn{Cyl}_{\mrm{Ke}}(\Z)$, and second tensor factor
the direct system of DG rings $(\til{C} \xar{v} C)$.
The system of DG rings $(\til{C} \xar{\opn{id}} \til{C})$ goes to the system
$(\til{C} \xar{v} C)$ by $(\opn{id} \to v)$, and this induces a DG ring
homomorphism
\[ w : \opn{Cyl}_{\mrm{Ke}}(\til{C}) \to D, \bmsp
w := \bmat{\opn{id}_{\til{C}} & y \ot v \\ 0 & \opn{id}_{\til{C}}} , \]
which is a surjective quasi-isomorphism.
Let
\[ u' : \til{B} \to D, \bmsp
u' := \bmat{\til{u}_{0} & 0 \\ 0 & \til{u}_{1}}  . \]
According to Theorem \ref{thm:157} there is a lifting
$\til{u}' : \til{B} \to \opn{Cyl}_{\mrm{Ke}}(\til{C})$
of $u'$ with respect to $w$; in a diagram:
\[ \begin{tikzcd} [column sep = 10ex, row sep = 4ex]
&
\opn{Cyl}_{\mrm{Ke}}(\til{C})
\ar[d, "{w}"]
\\
\til{B}
\ar[r, "{u'}" swap]
\ar[ur, dashed, "{\til{u}'}"]
&
D
\end{tikzcd} \]
In matrix notation we can express $\til{u}'$ as
\[ \til{u}' = \bmat{\til{u}_{0} & y \cd \ga \\ 0 & \til{u}_{1}}  \]
for a certain $A$-linear homomorphism $\ga : \til{B} \to \til{C}$ of degree
$-1$. Then
$\ga : \til{u}_0 \twoto \til{u}_1$
is a Keller homotopy, and $v \circ \ga = 0$.
\end{proof}

\begin{rem} \label{rem:185}
There is a small error in the proof of \cite[Theorem 3.22]{Ye1}.
Here is a description of the error, in terms of the proof of Theorem
\ref{thm:157} above. We had neglected to require that the element
$\til{c}''$ should satisfy $v(\til{c}'') = 0$.
Without this condition one can't guarantee that equation (\ref{eqn:186})
holds. This error has now been remedied in the proof of Theorem \ref{thm:157},
using Lemma \ref{lem:185}(2).

Lemma \ref{lem:185}(1) is needed because here the DG rings are not assumed
to be nonpositive. All we know here is that $\d(x) = 0$ for $x \in F_0(X)$.

We take this opportunity to list a few more small errors in the paper
\cite{Ye1}, and to fix them.

In Theorem 0.3.3 of that paper, and in the text immediately preceding it, one
needs to assume that the DG ring homomorphism $v : A' \to A$ is a {\em
quasi-isomorphism}. This very same correction, and its variants
$v' : A'' \to A'$ and $v_k : A_k \to A_{k  -1}$, have to be made in
several locations in that paper: Setup 5.7, Definition 5.8,
Setup 6.1, Definition 7.4 and Theorem 7.6.
This repeated error did
not effect Theorems 0.3.4 and 0.3.5 of \cite{Ye1}, since the quasi-isomorphism
was built into their assumptions (there is only one DG ring called $A$, without
primes or subscripts, so that $v$ is just the identity of $A$).

There are three typographic mistakes in \cite{Ye1} that we wish to point out.
In the third line of the proof of Theorem 3.22, the expression $X^i$ should be
replaced by $X^j$, so that the whole formula becomes
$F_i(X) = \bigcup_{-i \leq j \leq 0} X^j$.
Next, in Definition 5.6 the expression
$\opn{Hom}_{B^{\mrm{en}}}(B, \til{I})$
should be changed to
$\opn{Hom}_{\til{B}^{\mrm{en}}}(B, \til{I})$.
Finally, in the proof of Theorem 6.11, 5 lines before the end, the expression
$\opn{R}(y^{\dag} / r)$ should be changed to
$\opn{R}(s^{\dag} / r)$.
\end{rem}

\section{On the Structure of Horns}

This section is wholly about certain properties of simplicial sets.

First let's establish notation.
For a number $p \in \N$ let
$[p] := \{ 0, \ldots, p \}$.
The simplex category $\Simp$ has the ordered sets $[p]$ as objects, and the
nondecreasing functions
$\si : [p] \to [q]$ as morphisms.
A simplicial set is a functor
$X : \Simp^{\mrm{op}} \to \cat{Set}$.
A map of simplicial sets $f : X \to Y$ is by definition a a morphism of functors
$\Simp^{\mrm{op}} \to \cat{Set}$.
Thus the category of simplicial sets is
$\cat{SSet} = \cat{Fun}(\Simp^{\mrm{op}}, \cat{Set})$.
Given a simplicial set $X$, we usually write
$X = \{ X_{q} \}_{q \in \N}$, where $X_p := X([p])$.

For every $q$ there is the simplicial set
$\Simp^q = \{ \Simp^q_p \}_{p \in \N}$, which as a functor
$\Simp^{\mrm{op}} \to \cat{Set}$ equals $\opn{Hom}_{\Simp}(-, [q])$.
Its set of $p$-dimensional simplices is
$\Simp^q_p = \opn{Hom}_{\Simp}([p], [q])$.
An element $\si$ of $\Simp^q_p$ is the same as a nondecreasing sequence
$(i_0, \ldots, i_p)$ of elements of $[q]$.

The next fact will be used several times, so we state it as a proposition.
It is
\cite[Proposition 1.1.0.12, tag =
\href{https://kerodon.net/tag/04Z8}{\tt 04Z8}]{Lu2}.

\begin{prop} \label{prop:445}
For every simplicial set $X$ and every $p \in \N$ there is a canonical
isomorphism of sets
$X_p \cong \opn{Hom}_{\cat{SSet}}(\Simp^p, X)$.
\end{prop}

In particular, taking $X := \Simp^q$, we have
$\opn{Hom}_{\cat{SSet}}(\Simp^p, \Simp^q) =
\Simp^q_p = \opn{Hom}_{\Simp}([p], [q])$.

Consider the simplicial set $\Simp^q$ for some $q \in \N$.
For every $0 \leq i \leq q$ there is the horn
$\bsym{\La}^{q}_{i}$. This is the simplicial subset of $\Simp^q$
consisting of the elements (or simplices or nondecreasing functions)
$\si : [p] \to [q]$ such that
$\si([p]) \cup \{ i \}$ does not equal $[q]$.
See
\cite[Construction 1.2.4.1, tag =
\href{https://kerodon.net/tag/000U}{\tt 000U}]{Lu2}.

{\em For the rest of this section we fix a horn $\bsym{\La}^{q}_{i}$}.

Recall that for every number $j \in [q]$ we have the coface (or coboundary)
function $\pa^j : [q - 1] \to [q]$, which is the injective function in $\Simp$
whose image does not contain $j$.
It corresponds to the map of simplicial sets
$\pa^j : \Simp^{q - 1} \to \Simp^q$.
Thus $\pa^j(\Simp^{q - 1})$ is a simplicial subset of $\Simp^q$, isomorphic
to $\Simp^{q - 1}$ via $\pa^j$.

Let us define the set
\begin{equation} \label{eqn:190}
J_0 := [q] - \{ i \} = \{ 0, 1, \ldots, i - 1, i + 1, \ldots, q \} .
\end{equation}
For $j \in J_0$, the image of the coface map
$\pa^j : [q - 1] \to [q]$ must contain $i$. Since the cardinality of
$\pa^j([q - 1])$ is $q$,
the subset $\pa^j([q - 1])$ does not equal $[q]$. This implies that
the simplicial set $\pa^j(\Simp^{q - 1})$ is inside the horn
$\bsym{\La}^{q}_{i}$.
Thus, writing $\al_j := \pa^j$, we have an injective map of simplicial sets
\begin{equation} \label{eqn:193}
\al_j : \Simp^{q - 1} \to \bsym{\La}^{q}_{i} .
\end{equation}
Taking the disjoint union on all $j \in J_0$ we obtain a map of
simplicial sets
\begin{equation} \label{eqn:171}
\al : \coprod_{j \in J_0} \Simp^{q - 1} \to \bsym{\La}^{q}_{i} .
\end{equation}

Next define the set
\begin{equation} \label{eqn:191}
J_1 := \{ (k, l) \in J_0 \times J_0 \mid k < l \} .
\end{equation}
For every $(k, l) \in J_1$
let $\al_{k, l} : [q - 2] \to [q]$
be the injective function whose image does not contain $k$ and $l$.
The function $\al_{k, l}$ factors in precisely two ways through
$[q - 1]$~:
there is a unique injective function
$\be_{k, l} : [q - 2] \to [q - 1]$
such that $\al_{k, l} = \al_k \circ \be_{k, l}$, and a unique injective function
$\ga_{k, l} : [q - 2] \to [q - 1]$ such that
$\al_{k, l} = \al_l \circ \ga_{k, l}$.
It is pretty easy to see that
$\be_{k, l} = \pa^l$ and $\ga_{k, l} = \pa^k$.
The corresponding maps of simplicial sets are
\begin{equation} \label{eqn:194}
\be_{k, l} , \lmsp \ga_{k, l} : \Simp^{q - 2} \to \Simp^{q - 1} .
\end{equation}
Taking the coproducts on $J_0$ and on $J_1$ we obtain
the maps of simplicial sets
\begin{equation} \label{eqn:172}
\be, \ga : \coprod_{(k, l) \in J_1} \Simp^{q - 2} \to
\coprod_{j \in J_0} \Simp^{q - 1} .
\end{equation}

\begin{prop} \label{prop:170}
The following is a coequalizer sequence of simplicial sets:
\[ \begin{tikzcd}
\displaystyle\coprod_{(k, l) \in J_1}
\Simp^{q - 2} \mspace{10mu}
\arrow[r, shift left = 5pt, "{\be}"]
\arrow[r, shift right = 5pt, swap, "{\ga}"]
&
\mspace{10mu}
\displaystyle\coprod_{j \in J_0} \Simp^{q - 1}
\arrow[r, "{\al}"]
\mspace{10mu}
&
\mspace{10mu}
\bsym{\La}^{q}_{i}
\end{tikzcd} \]
\end{prop}

\begin{proof}
This, with slightly different notation,  is part of the proof of
\cite[Proposition 1.2.4.7, tag =
\href{https://kerodon.net/tag/050F}{\tt 050F}]{Lu2},
\end{proof}

The subcategory of $\Simp$ on all objects, but with only the injective functions
$\si : [p] \to [q]$ as morphisms, is denoted by
$\Simp_{\mrm{inj}}$.
A functor $X : \Simp^{\mrm{op}}_{\mrm{inj}} \to \cat{Set}$
is called a {\em semi-simplicial set}.
The category of semi-simplicial sets is
$\cat{S}_{\mrm{inj}} \mspace{-2mu} \cat{Set}
= \cat{Fun}(\Simp^{\mrm{op}}_{\mrm{inj}}, \cat{Set})$.

Let $X = \{ X_{p} \}_{p \in \N}$ be a simplicial set.
An element $x \in X_p$ is called {\em nondegenerate} if $x$ is not in the union
of the images of the degeneracy maps
$\opn{s}_j : X_{p - 1} \to X_p$.
Let us denote by $X^{\mrm{nd}}$ the set of nondegenerate elements of $X$.

If $X$ is a simplicial subset of $\Simp^q$ for some $q$, then
$x \in X_p$ is nondegenerate iff as a sequence
$x = (i_0, \ldots, i_p)$ in $[q]$, $x$ has no repetitions, i.e.\
$i_0 < \cdots < i_p$. Or equivalently, as a function
$x : [p] \to [q]$, $x$ is injective.
We can view $X^{\mrm{nd}}$ as a graded set.
For every injective function $\si : [r] \to [p]$ in $\Simp$, the function
$\si(x) = x \circ \si : [r] \to [q]$ is also injective, which means that the
simplex $\si(x) \in X_{r}$ is nondegenerate. We see that
$\si(X^{\mrm{nd}}_p) \sub X^{\mrm{nd}}_{r}$.
This means that $X^{\mrm{nd}}$ is in fact a semi-simplicial set.

\begin{lem} \label{lem:175}
Suppose
$\phi : \Simp^q \to \Simp^r$ is injective map of simplicial sets.
Let $x \in \Simp^q_p$ and  $y := \phi(x) \in \Simp^r_p$. Then $x$ is
nondegenerate iff $y$ is nondegenerate.
\end{lem}

\begin{proof}
Since $\Simp^q_0 \cong [q]$ and $\Simp^r_0 \cong [r]$, the map of simplicial
sets $\phi : \Simp^q \to \Simp^r$ is injective iff the corresponding function
$\phi_0 : [q] \to [r]$ in $\Simp$ is injective.
So we know that $\phi_0$ is injective.
As a sequence, or as a function $[p] \to [r]$ in $\Simp$, there is equality
$y = \phi_0 \circ x$.
We see that $y$ is injective iff $x$ is injective.
\end{proof}

By Lemma \ref{lem:175}, the maps of simplicial sets $\al_j$ in
(\ref{eqn:193}) satisfy
$\al_j((\Simp^{q - 1})^{\mrm{nd}}) \sub
(\bsym{\La}^{q}_{i})^{\mrm{nd}}$.
So there is a map
\begin{equation} \label{eqn:238}
\al : \coprod_{j \in J_0} (\Simp^{q - 1})^{\mrm{nd}} \to
(\bsym{\La}^{q}_{i})^{\mrm{nd}} .
\end{equation}

Likewise, the maps of simplicial sets $\be_{k, l}$ and $\ga_{k, l}$
in (\ref{prop:170}) satisfy
$\be_{k, l}((\Simp^{q - 2})^{\mrm{nd}}) \sub (\Simp^{q - 1})^{\mrm{nd}}$
and
$\ga_{k, l}((\Simp^{q - 2})^{\mrm{nd}}) \sub (\Simp^{q - 1})^{\mrm{nd}}$.
Taking coproducts we get
\begin{equation} \label{eqn:237}
\be, \ga : \coprod_{(k, l) \in J_1} (\Simp^{q - 2})^{\mrm{nd}} \to
\coprod_{j \in J_0} (\Simp^{q - 1})^{\mrm{nd}} .
\end{equation}

\begin{prop} \label{prop:171}
The following is a coequalizer sequence of semi-simplicial sets:
\[ \begin{tikzcd}
\displaystyle\coprod_{(k, l) \in J_1}
(\Simp^{q - 2})^{\mrm{nd}} \mspace{10mu}
\arrow[r, shift left = 5pt, "{\be}"]
\arrow[r, shift right = 5pt, swap, "{\ga}"]
&
\mspace{10mu}
\displaystyle\coprod_{j \in J_0} (\Simp^{q - 1})^{\mrm{nd}}
\arrow[r, "{\al}"]
\mspace{10mu}
&
\mspace{10mu}
(\bsym{\La}^{q}_{i})^{\mrm{nd}}
\end{tikzcd} \]
\end{prop}

\begin{proof}
Exactness at the last term means that
$\al : J_0 \times (\Simp^{q - 1})^{\mrm{nd}} \to
(\bsym{\La}^{q}_{i})^{\mrm{nd}}$
is surjective.
Take some $x \in (\bsym{\La}^{q}_{i})^{\mrm{nd}}$.
By Proposition \ref{prop:170} there is some $j \in J_0$ and
$y \in \Simp^{q - 1}$ such that $\al_j(y) = x$.
Because the map of simplicial sets
$\al_j : \Simp^{q - 1} \to  \Simp^{q}$ is injective, Lemma \ref{lem:175} says
that $y \in (\Simp^{q - 1})^{\mrm{nd}}$.
This settles exactness at the last term.

Exactness at the middle term means that if
$x, y \in J_0 \times (\Simp^{q - 1})^{\mrm{nd}}$ satisfy
$\al(x) = \al(y)$, then there is some
$z \in J_1 \times (\Simp^{q - 2})^{\mrm{nd}}$
such that
$\be(z) = x$ and $\ga(z) = y$.
By Proposition \ref{prop:170} there exist
$(k, l) \in J_1$ and $z \in \Simp^{q - 2}$
such that
$\be_{(k, l)}(z) = x$ and $\ga_{(k, l)}(z) = y$.
Because the map of simplicial sets
$\be_{(k, l)} : \Simp^{q - 2} \to  \Simp^{q - 1}$ is injective, Lemma
\ref{lem:175} says that $z \in (\Simp^{q - 2})^{\mrm{nd}}$.
This settles exactness at the middle term.
\end{proof}

For some calculations it will be more convenient to use colimits instead of
coequalizers.
Let us define the set $J := J_0 \coprod J_1$.
We make $J$ into a quiver, with arrows
$(k, l) \to k$ and $(k, l) \to l$.
Then we define the collection of simplicial sets
$X := \{ X_m \}_{m \in J}$ to be
$X_j := \Simp^{q - 1}$ for $j \in J_0$ and
$X_{(k, l)} := \Simp^{q - 2}$ for $(k, l) \in J_1$.
So we have maps of simplicial sets
$\be_{(k, l)} : X_{(k, l)} \to X_{k}$ and
$\ga_{(k, l)} : X_{(k, l)} \to X_{l}$.
In this way we make $X$ into a diagram
$X : J \to \cat{SSet}$.

\begin{cor} \label{cor:150}
The map $\al$ from \tup{(\ref{eqn:171})}
gives rise to an isomorphism
$\al : \colim{m \in J} X_m \iso \bsym{\La}^{q}_{i}$
in $\cat{SSet}$.
\end{cor}

\begin{proof}
This is merely a rephrasing of Proposition \ref{prop:170}.
\end{proof}

We can do the same for the nondegenerate variant.
Let us define the collection of semi-simplicial sets
$Y := \{ Y_m \}_{m \in J}$ to be
$Y_j := (\Simp^{q - 1})^{\mrm{nd}}$
for $j \in J_0$ and
$Y_{(k, l)} := (\Simp^{q - 2})^{\mrm{nd}}$ for $(k, l) \in J_1$.
So we have maps of semi-simplicial sets
$\be_{(k, l)} : Y_{(k, l)} \to Y_{k}$ and
$\ga_{(k, l)} : Y_{(k, l)} \to Y_{l}$.
In this way we make $Y$ into a diagram
$Y : J \to \cat{S}_{\mrm{inj}} \mspace{-2mu} \cat{Set}$.

\begin{cor} \label{cor:190}
The map $\al$ from \tup{(\ref{eqn:238})} gives rise to an isomorphism
$\al : \colim{m \in J} Y_m \iso (\bsym{\La}^{q}_{i})^{\mrm{nd}}$
in $\cat{S}_{\mrm{inj}} \mspace{-2mu} \cat{Set}$.
\end{cor}

\begin{proof}
This is a rephrasing of Proposition \ref{prop:171}.
\end{proof}

\section{Certain DG Rings Associated to Simplicial Sets}
\label{sec:certain-dg-rngs}

The coface, or coboundary, function
$\pa^j : [p] \to [p + 1]$ in $\Simp$ is the injective function that misses $j$.
As a sequence it is
\[ \pa^j  = (i_0, \ldots, i_p) = (0, \ldots, j - 1, j + 1, \ldots, p + 1) . \]
Here is a bit of nonstandard notation for some compositions of coboundary maps,
formulated in analogy to the notation $\pa^j$.
Let $\pa^{> p} : [p] \to [p + q]$ be the injective function in $\Simp$
that misses the subset $\{ i \in [p + q] \mid i > p\}$.
As a sequence it is
$\pa^{> p} = (0, \ldots, p)$.
Similarly let
$\pa^{< p} : [q] \to [p + q]$ by the function described by the sequence
$(p, \ldots, p + q)$.
For a simplicial object
$C = \{ C_p \}_{p \in \N}$
there are corresponding morphisms
$C(\pa^{> p}) : C_{p + q} \to C_p$ and
$C(\pa^{< p}) : C_{p + q} \to C_q$,
which we usually abbreviate to $\pa_{> p}$ and $\pa_{< p}$.
For a cosimplicial object
$D = \{ D^{p} \}_{p \in \N}$
there are corresponding morphisms
$\pa^{> p} : D^{p} \to D^{p + q}$ and
$\pa^{< p} : D^q \to D^{p + q}$.

Suppose $C = \{ C^p \}_{p \in \N}$ is a cosimplicial ring.
There are  coboundary operators
$\pa^i : C^p \to C^{p + 1}$
and codegeneracy operators
$\opn{s}^i : C^p \to C^{p - 1}$,
which are ring homomorphisms.
The normalization of $C$ is the DG ring
\begin{equation} \label{eqn:205}
\opn{N}(C) = \boplus_{p \in \N} \opn{N}^p(C) ,
\end{equation}
where
\begin{equation} \label{eqn:206}
\opn{N}^p(C) := \bigcap_i \opn{Ker}(\opn{s}^i) \sub C^p .
\end{equation}
The differential of $\opn{N}(C)$ is
\begin{equation} \label{eqn:207}
\d := \sum_i (-1)^i \cd \pa^i : \opn{N}^p(C) \to \opn{N}^{p + 1}(C) .
\end{equation}
The  Alexander-Whitney multiplication of
$b \in \opn{N}^p(C)$ and $c \in  \opn{N}^q(C)$ is
\begin{equation} \label{eqn:210}
b * c := \pa^{> p}(b) \cdot \pa^{< p}(c) \in C^{p + q} ,
\end{equation}
where $\cdot$ is the multiplication of the ring $C^{p + q}$.
Note then even if $C$ is a commutative cosimplicial ring, the DG ring
$\opn{N}(C)$ is almost never commutative.

Given a set $X$, there is the commutative ring
$\opn{Hom}_{\cat{Set}}(X, \Z)$.
Thus, from a simplicial set $X = \{ X_{p} \}_{p \in \N}$ we get the
cosimplicial ring
\begin{equation} \label{eqn:195}
\opn{Hom}_{\cat{Set}}(X, \Z) =
\{ \opn{Hom}_{\cat{Set}}(X_p, \Z) \}_{p \in \N} .
\end{equation}
The normalization of the cosimplicial DG ring
in (\ref{eqn:195}) is the DG ring
\begin{equation} \label{eqn:150}
\opn{N}(X, \Z) := \opn{N} \bigl(  \opn{Hom}_{\cat{Set}}(X, \Z) \bigr) .
\end{equation}

\begin{dfn} \label{dfn:195}
Let $X$ be a simplicial set and let $B$ be a DG ring. The
DG ring associated to $X$ and $B$ is the DG ring
\[ \opn{N}(X, B) := \opn{N}(X, \Z) \ot_{\Z} B . \]
\end{dfn}

This construction gives rise to a functor
\[ \cat{SSet}^{\mrm{op}} \times \cat{DGRng} \to \cat{DGRng} , \quad
(X, B) \mapsto  \opn{N}(X, B) . \]

Recall that $\Simp_{\mrm{inj}}$ is the subcategory of $\Simp$ on all the
objects, but the morphisms are only the injective (i.e.\ strictly increasing)
functions $\al : [p] \to  [q]$.
A semi-simplicial set is a functor
$Y : \Simp_{\mrm{inj}}^{\mrm{op}} \to \cat{Set}$.

Let $Y = \{ Y_{p} \}_{p \in \N}$ be a semi-simplicial set that is
dimension-wise finite, i.e.\ each $Y_p$ is a finite set.
For each $p$ we define the abelian group
\begin{equation} \label{eqn:208}
\opn{R}^p(Y, \Z) := \opn{Hom}_{\cat{Set}}(Y_p, \Z) .
\end{equation}
This is a free abelian group with basis consisting of the delta functions
$e_y$, $y \in Y_p$.
Taking the direct sum we obtain the graded abelian group
\begin{equation} \label{eqn:209}
\opn{R}(Y, \Z) = \bigoplus_{p \in \N} \opn{R}^p(Y, \Z) .
\end{equation}
For $y \in Y_p$ let
\[ Y^{+}_i(y) := \{ z \in Y_{p + 1} \mid \pa_i(z) = y \} . \]
Define the operator
\begin{equation} \label{eqn:201}
\d(e_y) := \sum_i (-1)^i \cd \sum_{z \in Y^{+}_i(y)} e_z .
\end{equation}
For $y \in Y_p$ and $z \in Y_q$ let
\[ Y^{+}(y, z) := \{ w \in Y_{p + q} \mid
\pa_{> p}(w) = y \ \text{and} \ \pa_{< p}(w) = z  \} . \]
Define a multiplication on $\opn{R}(Y, \Z)$ by
\begin{equation} \label{eqn:212}
e_y * e_z := \sum_{w \in Y^{+}(y, z)} e_w .
\end{equation}

\begin{prop} \label{prop:195}
Let $Y = \{ Y_{p} \}_{p \in \N}$ be a dimension-wise finite
semi-simplicial set. Then the graded abelian group
$\opn{R}(Y, \Z)$, with differential \tup{(\ref{eqn:201})} and multiplication
\tup{(\ref{eqn:212})}, is a DG ring.
\end{prop}

\begin{proof}
Let $\Z \ot_{\cat{Set}} Y_p$ be the free abelian group with basis
$\{ 1 \ot y \}_{y \in Y_p}$.
Consider the graded abelian group
\[ \Z \ot_{\cat{Set}} Y := \bigoplus_{p \in \N}
\Z \ot_{\cat{Set}} \lmsp Y_p . \]
It has the differential
\begin{equation} \label{eqn:215}
\d(1 \ot y) := \sum_i \lmsp (-1)^i \ot \pa^i(y)
\end{equation}
and the Alexander-Whitney comultiplication
\begin{equation} \label{eqn:216}
\opn{AW}(1 \ot y) :=
\sum_{i} \lmsp \bigl( 1 \ot \pa_{> i}(y) \bigr)
\ot \bigl( 1 \ot \pa_{< i}(y) \bigr) .
\end{equation}
These formulas make $\Z \ot_{\cat{Set}} Y$ into a DG coassociative coalgebra;
see \cite[Proposition VIII.8.7]{ML}.

There is a canonical isomorphism of graded abelian groups
\begin{equation} \label{eqn:217}
\opn{R}(Y, \Z) \cong \opn{Hom}_{\Z}
\bigl( \Z \ot_{\cat{Set}} Y, \Z \bigr) .
\end{equation}
A direct calculation shows that the differential (\ref{eqn:215}) and
comultiplication (\ref{eqn:216}) of
$\Z \ot_{\cat{Set}} Y$ go by this isomorphism to the differential
\tup{(\ref{eqn:201})} and the multiplication
\tup{(\ref{eqn:212})} of $\opn{R}(Y, \Z)$.
\end{proof}

\begin{dfn} \label{dfn:220}
Let $Y$ be a semi-simplicial set and let $B$ be a DG ring. The
DG ring associated to $Y$ and $B$ is
\[ \opn{R}(Y, B) := \opn{R}(Y, \Z) \ot_{\Z} B . \]
\end{dfn}

Note that since $Y$ is dimension-wise finite, there is a canonical isomorphism
of graded abelian groups
\begin{equation} \label{222}
\opn{R}(Y, B) \cong \opn{Hom}_{\cat{Set}}(Y, B) .
\end{equation}

Let  $X = \{ X_{p} \}_{p \in \N}$ be a dimension-wise finite simplicial set,
and let $Y := X^{\mrm{nd}}$, the semi-simplicial set of nondegenerate elements.
There is an inclusion of graded sets
$Y \sub X$, inducing a surjection of graded abelian groups
$\opn{Hom}_{\cat{Set}}(X, \Z) \to \opn{R}(Y, \Z)$,
see formula (\ref{eqn:208}). Also there is an inclusion of graded abelian
groups
$\opn{N}(X, \Z) \sub \opn{Hom}_{\cat{Set}}(X, \Z)$. These give rise to a
commutative diagram of graded abelian groups
\begin{equation} \label{eqn:214}
\begin{tikzcd} 
\opn{N}(X, \Z)
\ar[rr, bend left = 20, "{f}"]
\ar[r, tail, "{}"]
&
\opn{Hom}_{\cat{Set}}(X, \Z)
\ar[r, two heads, "{}"]
&
\opn{R}(Y, \Z)
\end{tikzcd}
\end{equation}

\begin{prop} \label{prop:210}
Let  $X = \{ X_{p} \}_{p \in \N}$ be a dimension-wise finite simplicial set,
and let $Y := X^{\mrm{nd}}$, the semi-simplicial set of nondegenerate elements.
Then the arrow $f$ in diagram \tup{(\ref{eqn:214})} is an isomorphism of
DG rings
\[ f : \opn{N}(X, \Z) \iso \opn{R}(Y, \Z) . \]
\end{prop}

\begin{proof}
The abelian subgroup
$\opn{N}^p(X, \Z)$ is free with basis the delta functions
$\{ e_{x} \}_{x \in X^{\mrm{nd}}_{p}}$.
The homomorphism $f$ is the identity on the delta functions
$e_x$, $x \in X^{\mrm{nd}}_p  = Y_p$.
Therefore $f$ is an isomorphism of graded abelian groups.
An elementary but  somewhat messy calculation shows that $f$ sends the
differential
(\ref{eqn:207}) to the differential (\ref{eqn:201}), and the AW multiplication
(\ref{eqn:210}) to the AW multiplication (\ref{eqn:212}).
\end{proof}

\begin{cor} \label{cor:240}
In the setting of Proposition \tup{\ref{prop:210}}, given a DG ring $B$,
the arrow $f$ in diagram \tup{(\ref{eqn:214})} induces an isomorphism of
DG rings
\[ f : \opn{N}(X, B) \iso \opn{R}(Y, B) . \]
\end{cor}

\begin{proof}
Clear from Proposition \ref{prop:210} and Definitions \ref{dfn:220} and
\ref{dfn:195}.
\end{proof}

\begin{dfn} \label{dfn:215}
Let $B$ be a DG ring.
\begin{enumerate}
\item For every $q \in \N$, the {\em $q$-th cylinder DG ring associated to
$B$} is the DG ring
\[ \opn{Cyl}_{q}(B) := \opn{N}(\Simp^q , B) =
 \opn{N}(\Simp^q, \Z) \ot_{\Z} B . \]

\item The {\em simplicial cylinder DG ring associated to $B$}
is
\[ \opn{Cyl}_{}(B) := \{ \opn{Cyl}_{q}(B) \}_{q \in \N} . \]
\end{enumerate}
\end{dfn}

\begin{prop} \label{prop:215}
There is a canonical isomorphism of DG rings between Keller's cylinder DG ring
$\opn{Cyl}_{\mrm{Ke}}(B)$
from Definition \tup{\ref{dfn:332}} and the DG ring $\opn{Cyl}_{1}(B)$.
\end{prop}

\begin{proof}
It is enough to consider the case $B = \Z$.

The structure of Keller's cylinder DG ring
$\opn{Cyl}_{\mrm{Ke}}(\Z) = \sbmat{\Z & \lmsp \Z[-1] \\ 0 & \Z}$
was recalled in Section \ref{sec:prelims-dg} above.
As a graded abelian group it is free, with basis
$e_0 = \sbmat{1 & \lmsp 0 \\ 0 & 0}$,
$e_1 = \sbmat{0 & \lmsp 0 \\ 0 & 1}$ and
$e_{0, 1} = \sbmat{0 & \lmsp y \\ 0 & 0}$.
The multiplication is that of matrices, and the differential is
$\d(e_0) = -e_{0, 1}$, $\d(e_1) = e_{0, 1}$ and
$\d(e_{0, 1}) = 0$.

The DG ring
$\opn{Cyl}_{1}(\Z) = \opn{N}(\Simp^1 , \Z) \cong
\opn{R}((\Simp^1)^{\mrm{nd}} , \Z)$
has a graded basis $e_{(0)}$, $e_{(1)}$ and $e_{(0, 1)}$.
Formula (\ref{eqn:201}) is very easy in this case:
$\d(e_{(0)}) = - e_{(0, 1)}$,
$\d(e_{(1)}) = e_{(0, 1)}$, and
$\d(e_{(0, 1)}) = 0$.
And formula (\ref{eqn:212}) says that
$e_{(0)} * e_{(0)} = e_{(0)}$,
$e_{(1)} * e_{(1)} = e_{(1)}$,
$e_{(0)} * e_{(0, 1)} = e_{(0, 1)}$,
$e_{(0, 1)} * e_{(1)} = e_{(0, 1)}$,
and all other multiplications are zero.

We see that the graded abelian group isomorphism
$e_{0} \mapsto e_{(0)}$, $e_{1} \mapsto e_{(1)}$ and
$e_{0, 1} \mapsto e_{(0, 1)}$ is an isomorphism of DG rings.
\end{proof}

\section{The DG Ring Associated to a Horn}
\label{sec:dg-ring-horn}

In this section we prove the key technical result, Theorem \ref{thm:245} from
the Introduction, which talks about the role of the DG ring
$\opn{N}(\bsym{\La}^{q}_{i}, B)$ associated to a horn
$\bsym{\La}^{q}_{i} \sub \Simp^q$ and a DG ring $B$. It is repeated here as
Theorem \ref{thm:152}.

Consider a horn $\bsym{\La}^{q}_{i}$.
In Corollary \ref{cor:150} we expressed this horn as a colimit
\begin{equation} \label{eqn:220}
\bsym{\La}^{q}_{i} \cong \colim{j \in J} X_j ,
\end{equation}
in the category $\cat{SSet}$, of a diagram $\{ X_j \}_{j \in J}$,
indexed by a finite quiver $J$, where each
$X_j = \Simp^{q_j}$ for some natural number $q_j$.

Passing to nondegenerate simplices, let
$Y := (\bsym{\La}^{q}_{i})^{\mrm{nd}}$
and $Y_j := X_j^{\mrm{nd}}$.
These are semi-simplicial sets, i.e.\ objects of the category
$\cat{S}_{\mrm{inj}} \mspace{-2mu} \cat{Set}$.
In Corollary \ref{cor:190} we proved that (\ref{eqn:220}) induces a
canonical isomorphism
\begin{equation} \label{eqn:225}
Y \cong \colim{j \in J} Y_j
\end{equation}
in $\cat{S}_{\mrm{inj}} \mspace{-2mu} \cat{Set}$.

{\em The horn $\bsym{\La}^{q}_{i}$ and the direct systems
$\{ X_j \}_{j \in J}$ and $\{ Y_j \}_{j \in J}$ are retained
until the end of the proof of Theorem \tup{\ref{thm:152}}}.

To avoid confusion (e.g.\ in formulas (\ref{eqn:226}) and (\ref{eqn:223})
below), for a simplicial set
$Z = \{ Z_{p} \}_{p \in \N}$
we will often write $Z([p])$ instead of $Z_p$.
(This is actually the proper notation, since $Z$ is a functor
$\Simp^{\mrm{op}} \to \cat{Set}$~; $Z_p$ is just an abbreviation.)
Likewise for semi-simplicial sets.

Since $\{ Y_j \}_{j \in J}$ is a direct system of functors
$\Simp^{\mrm{op}}_{\mrm{inj}} \to \cat{Set}$,
the direct limit is calculated in the target category.
This means that in each simplicial dimension $p$ we have
\begin{equation} \label{eqn:226}
Y([p]) \cong \colim{j \in J} Y_j([p])
\end{equation}
in $\cat{Set}$.

Let $B$ be a DG ring.
By Definition \ref{dfn:215} we have
$\opn{Cyl}_{q_j}(B) = \opn{N}(\Simp^{q_j} , B)$.
Therefore the direct system of simplicial sets
$\{ \Simp^{q_j} \}_{j \in J} = \{ X_j \}_{j \in J}$
gives rise to an inverse system of DG rings
$\{ \opn{Cyl}_{q_j}(B) \}_{j \in J}$.

\begin{lem} \label{lem:152}
Let $B$ be any DG ring. Then there is a canonical isomorphism
\[ \opn{N}(\bsym{\La}^{q}_{i}, B) \cong
\underset{j \in J}{\opn{lim}} \, \opn{Cyl}_{q_j}(B) \]
in $\cat{DGRng}$.
\end{lem}

\begin{proof}
For the sake of clarity let's write
$X := \bsym{\La}^{q}_{i}$,
$Y := X^{\mrm{nd}} = (\bsym{\La}^{q}_{i})^{\mrm{nd}}$
and $Y_j := X_j^{\mrm{nd}} = (\Simp^{q_j})^{\mrm{nd}}$.

By Proposition \ref{prop:210} we can replace
$\opn{N}(X, B)$ with $\opn{R}(Y, B)$, and
$\opn{Cyl}_{q_j}(B) = \opn{N}(X_j, B)$ with $\opn{R}(Y_j, B)$.
We now need to prove that the canonical homomorphism
\begin{equation} \label{eqn:224}
\opn{R}(Y, B) \to \underset{j \in J}{\opn{lim}} \, \opn{R}(Y_j, B)
\end{equation}
in $\cat{DGRng}$ is an isomorphism.
The forgetful functor from DG rings to graded abelian
groups commutes with taking limits.
Hence we can forget the DG ring structure,
and just look at equation (\ref{eqn:224}) in the category of graded
abelian groups.

By definition there are equalities
$\opn{R}(Y, B) = \opn{R}(Y, \Z) \ot_{\Z} B$,
$\opn{R}(Y, \Z) = \bigoplus_{p} \opn{R}^p(Y, \Z)$,
and
$\opn{R}^p(Y, \Z) = \opn{Hom}_{\cat{Set}}(Y([p]), \Z)$.
The DG ring $B$ decomposes, as a graded abelian group, into
$B = \bigoplus_{k} B^k$.
Because the set $Y([p])$ finite, there is a canonical isomorphism of
abelian groups
\[ \opn{Hom}_{\cat{Set}}(Y([p]), \Z) \ot_{\Z} B^k \cong
\opn{Hom}_{\cat{Set}}(Y([p]), B^k) . \]
All this holds for the $Y_j$ too. Therefore we can replace
(\ref{eqn:224}) with the homomorphism
\begin{equation} \label{eqn:223}
\opn{Hom}_{\cat{Set}}(Y([p]), B^k) \to
\underset{j \in J}{\opn{lim}}  \opn{Hom}_{\cat{Set}}(Y_j([p]), B^k) ,
\end{equation}
where limit is in $\cat{Ab}$. But this limit coincides, on underlying sets, with
the limit in $\cat{Set}$. And in $\cat{Set}$ the map (\ref{eqn:223})
is an isomorphism by formula (\ref{eqn:226}).
\end{proof}

\begin{lem} \label{lem:154}
Let $A$ and  $B$ be DG rings. Then for any $j \in J$ there is canonical
isomorphism
\[ \opn{Hom}_{\cat{SSet}}
\bigl( X_j, \opn{Hom}_{\cat{DGRng}}(A, \opn{Cyl}_{}(B)) \bigr) \cong
\opn{Hom}_{\cat{DGRng}}(A, \opn{Cyl}_{q_j}(B)) \]
in $\cat{Set}$.
\end{lem}

\begin{proof}
Recall that $X_j = \Simp^{q_j}$.
Now apply Proposition \ref{prop:445} to the simplicial set
\[ \opn{Hom}_{\cat{DGRng}}(A, \opn{Cyl}_{}(B)) =
\{ \opn{Hom}_{\cat{DGRng}}(A, \opn{Cyl}_{p}(B)) \}_{p \in \N} . \qedhere \]
\end{proof}

Here is our key technical result.

\begin{thm} \label{thm:152}
Let $A$ and $B$ be DG rings, and let $\bsym{\La}^{q}_{i}$ be a horn in
$\Simp^q$. Then there is a canonical bijection
\[ \opn{Hom}_{\cat{SSet}} \bigl( \bsym{\La}^{q}_{i},
\opn{Hom}_{\cat{DGRng}}(A, \opn{Cyl}_{}(B)) \bigr) \cong
\opn{Hom}_{\cat{DGRng}}(A, \opn{N}(\bsym{\La}^{q}_{i}, B)) . \]
\end{thm}

\begin{proof}
We have this sequence of canonical bijections:
\[ \begin{aligned}
&
\opn{Hom}_{\cat{SSet}} \bigl( \bsym{\La}^{q}_{i},
\opn{Hom}_{\cat{DGRng}}(A, \opn{Cyl}_{}(B)) \bigr)
\\
& \quad \cong^{(1)}
\opn{Hom}_{\cat{SSet}} \bigl( \colim{j \in J} X_j,
\opn{Hom}_{\cat{DGRng}}(A, \opn{Cyl}_{}(B)) \bigr)
\\
& \quad \cong^{(2)}
\mylim{j \in J} \opn{Hom}_{\cat{SSet}} \bigl(  X_j,
\opn{Hom}_{\cat{DGRng}}(A, \opn{Cyl}_{}(B)) \bigr)
\\
& \quad \cong^{(3)}
\mylim{j \in J} \opn{Hom}_{\cat{DGRng}}(A, \opn{Cyl}_{q_j}(B))
\\
& \quad \cong^{(4)}
\opn{Hom}_{\cat{DGRng}}(A, \mylim{j \in J} \opn{Cyl}_{q_j}(B))
\\
& \quad \cong^{(5)}
\opn{Hom}_{\cat{DGRng}}(A, \opn{N}(\bsym{\La}^{q}_{i}, B)) .
\end{aligned} \]
The isomorphism $\cong^{(1)}$ is by Corollary \ref{cor:150}.
The isomorphism $\cong^{(2)}$ is by the definition of a
colimit in the category $\cat{SSet}$.
The isomorphism $\cong^{(3)}$ is by Lemma \ref{lem:154}.
The isomorphism $\cong^{(4)}$ is by the definition of a
limit in the category
$\cat{DGRng}$.
And the isomorphism $\cong^{(5)}$ is by Lemma \ref{lem:152}.
\end{proof}

Recall that a horn in a simplicial set $Z$ is a map of simplicial sets
$\si : \La^q_i \to Z$ from a horn $\La^q_i \sub \Simp^q$.
The next corollary is merely a reformulation of Theorem \ref{thm:152}.

\begin{cor} \label{cor:152}
There is a canonical bijection between horns
\[ \si : \La^q_i \to \opn{Hom}_{\cat{DGRng}}(A, \opn{Cyl}_{}(B)) \]
and DG ring homomorphisms
\[ f : A \to \opn{N}(\La^q_i, B) . \]
\end{cor}

\section{The Kan Condition and the Hom Groupoid}
\label{sec:kan-cond}

In this section we prove the main theorem of the paper, i.e.\ Theorem
\ref{thm:150} from the Introduction. It is repeated here as Theorem
\ref{thm:215}. After that we provide a detailed description of the
resulting fundamental groupoid, from the simplicial angle.

Recall the simplicial cylinder  DG ring
$\opn{Cyl}_{}(B) := \{ \opn{Cyl}_{q}(B) \}_{q \in \N}$
from Definition \ref{dfn:215}, associated to a DG ring $B$.

\begin{dfn} \label{dfn:300}
Let $A$ and $B$ be DG rings.
\begin{enumerate}
\item For every $q \in \N$ let
\[ \opn{SHom}_q(A, B) :=
\opn{Hom}_{\cat{DGRng}}(A, \opn{Cyl}_{q}(B)) . \]
It is called the set of {\em $q$-dimensional DG ring homomorphisms from $A$ to
$B$}.

\item The simplicial set
\[ \opn{SHom}(A, B) := \bigl \{ \opn{SHom}_q(A, B) \bigr\}_{q \in \N} \]
is called the {\em simplicial Hom set of DG ring homomorphisms from $A$ to
$B$}.
\end{enumerate}
\end{dfn}

The simplicial structure on $\opn{SHom}(A, B)$ comes from that of
$\opn{Cyl}_{}(B)$.

Here is the main theorem of our paper (a repetition of Theorem \ref{thm:150}
from the Introduction).

\begin{thm} \label{thm:215}
Let $\til{A}$ be a semi-free DG ring, and let $B$ be any DG ring.
Then the simplicial set
$\opn{SHom}(\til{A}, B)$
is a Kan complex.
\end{thm}

The proof of the theorem requires a lemma.

\begin{lem} \label{lem:157}
Let
$\si : \La^q_i \to \Simp^q$ be the inclusion of a horn.
The the induced DG ring homomorphism
\[ v : \opn{N}(\Simp^q, B) \to \opn{N}(\La^q_i, B) \]
is a surjective quasi-isomorphism.
\end{lem}

\begin{proof}
In view of Definition \ref{dfn:195}, and because
the DG rings $\opn{N}(\Simp^q, \Z)$ and
$\opn{N}(\La^q_i, \Z)$ are semi-free as DG $\Z$-modules,
we may assume that $B = \Z$.

Both simplicial sets $\Simp^q$ and $\La^q_i$ are weakly contractible, and hence
the DG ring homomorphisms
$\Z \to \opn{N}(\Simp^q, \Z)$ and
$\Z \to \opn{N}(\La^q_i, \Z)$
are quasi-isomorphisms. Therefore $v$ is a quasi-isomorphism.

According to  Proposition \ref{prop:210} we can replace the DG ring homomorphism
\[ v : \opn{N}(\Simp^q, \Z) \to \opn{N}(\La^q_i, \Z) \]
with the isomorphic
\[ w : \opn{R}((\Simp^q)^{\mrm{nd}}, \Z) \to
\opn{R}((\La^q_i)^{\mrm{nd}}, \Z) . \]
As a homomorphism of graded abelian groups, this is
\[ w' : \opn{Hom}_{\cat{Set}}((\Simp^q)^{\mrm{nd}}, \Z) \to
\opn{Hom}_{\cat{Set}}((\La^q_i)^{\mrm{nd}}, \Z) . \]
Since
$\si' : (\La^q_i)^{\mrm{nd}} \to (\Simp^q)^{\mrm{nd}}$
is an injective map of graded sets, and
$w' = \opn{Hom}_{\cat{Set}}(\si', \Z)$,
it follows that $w'$ is a surjective homomorphism of graded abelian groups.
\end{proof}

\begin{proof}[Proof of Theorem \tup{\ref{thm:215}}]
Let
\[ \si : \La^{q}_{i}  \to \opn{Hom}_{\cat{DGRng}}(\til{A}, \opn{Cyl}_{}(B))
= \opn{SHom}_{}(\til{A}, B) \]
be a horn. According to Corollary \ref{cor:152}, $\si$ corresponds canonically
to a DG ring homomorphism
$f : \til{A} \to \opn{N}(\La^{q}_{i}, B)$.
By Lemma \ref{lem:157} the DG ring homomorphism
$v : \opn{N}(\Simp^{q}, B) \to \opn{N}(\La^{q}_{i}, B)$
is a surjective quasi-isomorphism.
Theorem \ref{thm:157} says that there is a DG ring homomorphism
$f' : \til{A} \to  \opn{N}(\Simp^{q}, B)$
lifting $f$. See next diagram.
\[ \begin{tikzcd} [column sep = 10ex, row sep = 4ex]
&
\opn{N}(\Simp^{q}, B)
\ar[d, "{v}"]
\\
\til{A}
\ar[r, "{f}" swap]
\ar[ur, dashed, "{f'}"]
&
\opn{N}(\La^{q}_{i}, B)
\end{tikzcd} \]
Thus
\[ f' \in \opn{Hom}_{\cat{DGRng}}(\til{A}, \opn{Cyl}_{q}(B))
= \opn{SHom}_{q}(\til{A}, B) . \]
By Proposition \ref{prop:445} there is a corresponding map of
simplicial sets
\[ \si' : \Simp^{q}  \to \opn{Hom}_{\cat{DGRng}}(\til{A}, \opn{Cyl}_{}(B))
= \opn{SHom}_{}(\til{A}, B)  . \]
A little calculation shows that $\si'$ is a filler for $\si$.
\end{proof}

In the remainder of this section we discuss an important consequence of Theorem
\ref{thm:215}.
We fix a semi-free DG ring $\til{A}$, and an arbitrary DG ring $B$.
It will be
convenient to introduce the abbreviation
\begin{equation} \label{eqn:425}
S_q := \opn{SHom}_q(\til{A}, B) =
\opn{Hom}_{\cat{DGRng}}(\til{A}, \opn{Cyl}_{q}(B)) .
\end{equation}
Thus the simplicial set $S := \{ S_{q} \}_{q \in \N}$ is a Kan complex.
According to
\cite[Definition 1.4.6.12, tag =
\href{https://kerodon.net/tag/00HZ}{\tt 00HZ}]{Lu2},
the Kan complex $S$ has a {\em fundamental groupoid}
$\bpi_{\leq 1}(S)$.

\begin{dfn} \label{dfn:400}
Let $\til{A}$ be a semi-free DG ring, and let $B$ be any DG ring.
The
{\em groupoid of DG ring homomorphisms from $\til{A}$ to $B$}, also called
the {\em Hom groupoid from $\til{A}$ to $B$},
is the groupoid
\[ \opn{SHom}_{\leq 1}(\til{A}, B) :=
\bpi_{\leq 1} \bigl( \opn{SHom}(\til{A}, B) \bigr) =
\bpi_{\leq 1} \bigl( \opn{Hom}_{\cat{DGRng}}(\til{A}, \opn{Cyl}_{}(B)) \bigr)
. \]
\end{dfn}

Here is a detailed description of the groupoid
$S_{\leq 1} = \bpi_{\leq 1}(S)$ in terms of the simplicial data,
following
\cite[Section 1.4, tag = \href{https://kerodon.net/tag/0039}{\tt 0039}]{Lu2},
but translated to our setting.
The objects of $S_{\leq 1}$ are the elements of $S_0$, namely the DG ring
homomorphisms
$f : A \to B$. They will be denoted by the letter $f$, sometimes with
subscripts. The elements of $S_1$ are DG ring homomorphisms
$g : A \to \opn{Cyl}_{1}(B)$.
The simplicial structure of $S$ indicates that $g$ represents a $1$-morphism
from $f_{(0)} := \pa_1(g)$ to $f_{(1)} := \pa_0(g)$.
In diagrams it looks looks this:
\begin{equation} \label{eqn:395}
\begin{tikzcd} [column sep = 6ex, row sep = 6ex]
(0)
\ar[r, "{(0, 1)}"]
&
(1)
\end{tikzcd}
\qquad
\begin{tikzcd} [column sep = 6ex, row sep = 6ex]
f_{(0)}
\ar[r, "{g}"]
&
f_{(1)}
\end{tikzcd}
\end{equation}
The first diagram shows the nondegenerate elements of $\Simp^1$, and the second
shows the corresponding elements of $S$.

Fixing $f_{(0)}$ and $f_{(1)}$, let us denote by
$S_1(f_{(0)}, f_{(1)})$ the set of all the $g \in S_1$ such that
$f_{(0)} = \pa_1(g)$ and $f_{(1)} = \pa_0(g)$.
For $f_{(0)} = f_{(1)} = f$, the set $S_1(f, f)$ contains the special element
$\opn{id}_f := \opn{s}_0(f)$.

The elements $h \in S_2$ are the DG ring homomorphisms
$h : \til{A} \to \opn{Cyl}_2(B)$. Their internal algebraic description will be
studied in Section \ref{sec:algeb-high-homs}.
Their simplicial roles are explained here.

Suppose we are given $g_{(0, 1)}, g_{(0, 2)} \in S_1(f_{(0)}, f_{(1)})$.
A {\em simplicial homotopy}  from $g_{(0, 1)}$ to $g_{(0, 2)}$ is an element
$h \in S_2$ such that
$\pa_0(h) = \opn{id}_{f_{(1)}}$,
$\pa_1(h) = g_{(0, 2)}$ and $\pa_2(h) = g_{(0, 1)}$.
This is shown in the next diagrams.
\begin{equation} \label{eqn:355}
\begin{tikzcd} [column sep = 6ex, row sep = 6ex]
&
(2)
\arrow[d, phantom, "{(0, 1, 2)}"  {scale=0.9}]
\\
(0)
\ar[ur, "{(0, 2)}"]
\ar[rr, "{(0, 1)}" swap]
&
{}
&
(1)
\ar[ul, "{(1, 2)}" swap]
\end{tikzcd}
\qquad
\begin{tikzcd} [column sep = 6ex, row sep = 6ex]
&
f_{(2)}
\arrow[d, phantom, "{h}"]
\\
f_{(0)}
\ar[ur, "{g_{(0, 2)}}"]
\ar[rr, "{g_{(0, 1)}}" swap]
&
{}
&
f_{(1)}
\ar[ul, "{\opn{id}_{f_{(1)}}}" swap]
\end{tikzcd}
\end{equation}
This is a rephrasing of
\cite[Definition 1.4.3.1, tag =
\href{https://kerodon.net/tag/0037}{\tt 0037}]{Lu2}.
According to
\cite[Proposition 1.4.3.5, tag =
\href{https://kerodon.net/tag/003Z}{\tt 003Z}]{Lu2},
simplicial homotopy is an equivalence relation on the set
$S_1(f_{(0)}, f_{(1)})$.
For $g \in S_1(f_{(0)}, f_{(1)})$ let us denote by $[g]$ its simplicial homotopy
class.

Next we introduce {\em simplicial composition}, following
\cite[Definition 1.4.4.1, tag = \lb
\href{https://kerodon.net/tag/0042}{\tt 0042}]{Lu2}.
Suppose we are given
$f_{(0)}, f_{(1)}, f_{(2)} \in S_0$ and
$g_{(i, j)} \in S_1(f_{(i)}, f_{(j)})$.
The element $g_{(0, 2)}$ is said to be a simplicial composition of
$g_{(0, 1)}$ and $g_{(1, 2)}$ if there exists some
$h \in S_2$ such that
$\pa_0(h) = g_{(1, 2)}$,
$\pa_1(h) = g_{(0, 2)}$ and
$\pa_2(h) = g_{(0, 1)}$.
See next diagrams.
\begin{equation} \label{eqn:392}
\begin{tikzcd} [column sep = 6ex, row sep = 6ex]
&
(2)
\arrow[d, phantom, "{(0, 1, 2)}"  {scale=0.9}]
\\
(0)
\ar[ur, "{(0, 2)}"]
\ar[rr, "{(0, 1)}" swap]
&
{}
&
(1)
\ar[ul, "{(1, 2)}" swap]
\end{tikzcd}
\qquad
\begin{tikzcd} [column sep = 6ex, row sep = 6ex]
&
f_{(1)}
\arrow[d, phantom, "{h}"]
\\
f_{(0)}
\ar[ur, "{g_{(0, 2)}}"]
\ar[rr, "{g_{(0, 1)}}" swap]
&
{}
&
f_{(1)}
\ar[ul, "{g_{(1, 2)}}" swap]
\end{tikzcd}
\end{equation}
Note that (\ref{eqn:355}) is a special case of (\ref{eqn:392}).
According to
\cite[Proposition 1.4.4.2, tag =
\href{https://kerodon.net/tag/0043}{\tt 0043}]{Lu2},
simplicial composition is well-defined on homotopy classes,
i.e.\ the multiplication
$[g_{(0, 2)}] := [g_{(1, 2)}] \cd [g_{(0, 1)}]$
does depend on representatives.
Moreover, the composition is associative, with units $[\opn{id}_f]$, and all the
$[g]$ are invertible for composition.
Thus we obtain the groupoid $\bpi_{\leq 1}(S)$.

\section{Algebraic Description of Higher Homomorphisms}
\label{sec:algeb-high-homs}

In this section we describe the algebraic properties of higher DG
ring homomorphisms.
The main result is Theorem \ref{thm:390}, which is a higher ($q = 2$)
variant of Keller's description of DG ring homomorphisms
$g : A \to \opn{Cyl}_{1}(B)$ in terms of Keller derivations.

Throughout this section we fix DG rings $A$ and $B$. (We do not assume $A$ is
semi-free.)
For every $q$ we have the set
\begin{equation} \label{eqn:400}
S_q := \opn{Hom}_{\cat{DGRng}}(A, \opn{Cyl}_{q}(B)) .
\end{equation}
As $q$ varies we obtain the simplicial set
$S = \{ S_{q} \}_{q \in \N}$.
Note that $S$ might not be a Kan complex.

The elements of $S_0$ are the DG ring homomorphisms
$f : A \to B$. There is nothing more to say about them.

The elements of $S_1$ are the DG ring homomorphisms
$g : A \to \opn{Cyl}_1(B)$. As explained in Section
\ref{sec:prelims-dg}, these correspond to
Keller homotopies $\ga : f_{(0)} \twoto f_{(1)}$. The simplicial relation
between $g$, $f_{(0)}$ and $f_{(1)}$ is this:
$f_{(0)} = \pa_1(g)$ and $f_{(1)} = \pa_0(g)$.
The matrix description is
\[ g = \bmat{f_{(0)} & \bmsp y \cd \ga \\ 0 & f_{(1)}} , \]
where $y$ is a degree $1$ variable.
In terms of the standard basis
$e_{(0)}, e_{(1)}, e_{(0, 1)}$ of $\opn{Cyl}_1(\Z)$, we have
\[ g =
\bigl( e_{(0)} \ot f_{(0)} \bigr) + \bigl( e_{(1)} \ot f_{(1)} \bigr) +
\bigl( e_{(0, 1)} \ot \ga \bigr) . \]

The elements of $S_2$ are the higher DG ring   homomorphisms
$h : A \to \opn{Cyl}_2(B)$.
Before describing $h$, we need to know more about the  DG ring
$\opn{Cyl}_{2}(B)$.

The DG ring $\opn{Cyl}_{2}(B)$ can't be described as a matrix DG ring,
so we will have to describe its elements and its operations by formulas.
Recall that
\begin{equation} \label{eqn:335}
\opn{Cyl}_{2}(B) = \opn{Cyl}_{2}(\Z) \ot_{\Z} B =
\boplus_{0 \leq p \leq 2}
\opn{Cyl}_{2}(\Z)^p \ot_{\Z} B ,
\end{equation}
a special case of Proposition \ref{prop:425}.

By definition we have
$\opn{Cyl}_{2}(\Z) = \opn{N}(\Simp^2, \Z)$.
Proposition \ref{prop:210} says that there is a canonical DG ring isomorphism
$\opn{Cyl}_{2}(\Z) \cong \opn{R}((\Simp^2)^{\mrm{nd}}, \Z)$.
In Proposition \ref{prop:195} the DG ring structure of
$\opn{R}((\Simp^2)^{\mrm{nd}}, \Z)$
is stated implicitly. Below we are going to write out this structure in detail.

The standard basis of $\opn{Cyl}_{2}(\Z)$ as a graded
abelian group is
\begin{equation} \label{eqn:345}
\begin{aligned}
& e_{(0)}, e_{(1)}, e_{(2)}
\\ &
e_{(0, 1)}, e_{(0, 2)}, e_{(1, 2)}
\\ &
e_{(0, 1, 2)} ,
\end{aligned}
\end{equation}
indexed by the elements of $(\Simp^2)^{\mrm{nd}}$.
The basis elements in the first row have degree $0$,
the elements in the second row have degree $1$,
and the element in the third row has degree $2$.
A homogeneous element $b \in \opn{Cyl}_{2}(B)^i$ is uniquely a sum
$b = \sum_{\bsym{j}}  e_{\bsym{j}} \ot b_{\bsym{j}}$,
with coefficients $b_{\bj} \in B^{i - p}$ for
$\bj = (j_0, \ldots, j_p)$.
In detail:
\begin{equation} \label{eqn:337}
\begin{aligned}
b
& = \bigl( e_{(0)} \ot b_{(0)} \bigr) +
\bigl( e_{(1)} \ot b_{(1)} \bigr) +
\bigl( e_{(2)} \ot b_{(2)} \bigr)
\\ & \quad
+ \bigl( e_{(0, 1)} \ot b_{(0, 1)} \bigr)
+ \bigl( e_{(0, 2)} \ot b_{(0, 2)} \bigr)
+ \bigl( e_{(1, 2)} \ot b_{(1, 2)} \bigr)
\\ & \quad
+ \bigl( e_{(0, 1, 2)} \ot b_{(0, 1, 2)} \bigr) .
\end{aligned}
\end{equation}

Here are the differentials of the basis elements from equation (\ref{eqn:345}):
\begin{equation} \label{eqn:336}
\begin{aligned}
&
\d(e_{(0)}) =  -e_{(0, 1)} - e_{(0, 2)} , \bmsp
\d(e_{(1)}) =  e_{(0, 1)} - e_{(1, 2)} , \bmsp
\d(e_{(2)}) =  e_{(0, 2)} + e_{(1, 2)}
\\
&
\d(e_{(0, 1)}) =  e_{(0, 1, 2)} , \bmsp
\d(e_{(0, 2)}) =  - e_{(0, 1, 2)} , \bmsp
\d(e_{(1, 2)}) =  e_{(0, 1, 2)}
\\
&
\d(e_{(0, 1, 2)}) = 0 .
\end{aligned}
\end{equation}
Hence the differential of a homogeneous element
$b = e_{\bj} \ot b_{\bj} \in \opn{Cyl}_{2}(\Z)^i$ is
\begin{equation}
\d(b) =
\bigl( \d(e_{\bj}) \ot b_{\bj} \bigr)) +
(-1)^p \cd \bigl( e_{\bj} \ot \d_{B}(b_{\bj} \bigr)
\bigr)
\end{equation}
for $\bj = (j_0, \ldots, j_p)$.

Addition is straightforward: for
$b = \sum_{\bsym{i}}  e_{\bsym{i}} \ot b_{\bsym{i}}$
and
$b' = \sum_{\bsym{i}}  e_{\bsym{i}} \ot b'_{\bsym{i}}$
the sum is
\begin{equation} \label{eqn:346}
c + c' =
\sum_{\bsym{i}}  e_{\bsym{i}} \ot (b_{\bsym{i}} + b'_{\bsym{i}}) .
\end{equation}

For multiplication we need to first specify how the standard basis elements
multiply. This is by the AW rule, see equation (\ref{eqn:212}). For sequences
$\bi = (i_0, \ldots, i_p)$ and
$\bj = (j_0, \ldots, j_q)$
in $(\Simp^2)^{\mrm{nd}}$
such that $i_p = j_0$, their concatenation is
\[ \bi \smallsmile \bj :=
(i_0, \ldots, i_p = j_0, j_1, \ldots, j_q) . \]
The multiplication is then:
\begin{equation} \label{eqn:339}
e_{\bi} \cd e_{\bj} =
\begin{cases}
e_{\bi \smallsmile \bj} & \text{if} \bmsp i_p = j_0
\\
0 & \text{otherwise} .
\end{cases}
\end{equation}
Therefore the multiplication of homogeneous elements
of $\opn{Cyl}_{2}(B)$, with $\bi$ and $\bj$ as above,
and with 
$b_{\bi} \in B^{p_{\bi}}$, is
\begin{equation} \label{eqn:340}
(e_{\bi} \ot b_{\bi}) \cd (e_{\bj} \ot b_{\bj}) =
(-1)^{p_{\bi} \cd q} \cd (e_{\bi} \cd e_{\bj}) \ot (b_{\bi} \cd b_{\bj}) .
\end{equation}

A graded abelian group homomorphism
$h : A \to \opn{Cyl}_{2}(B)$ of degree $0$ is described uniquely by such a sum:
\begin{equation} \label{eqn:360}
\begin{aligned}
h & =
\bigl( e_{(0)} \ot f_{(0)} \bigr) +
\bigl( e_{(1)} \ot f_{(1)} \bigr) +
\bigl( e_{(2)} \ot f_{(2)} \bigr)
\\ & \quad
+ \bigl( e_{(0, 1)} \ot \ga_{(0, 1)} \bigr)
+ \bigl( e_{(0, 2)} \ot \ga_{(0, 2)} \bigr)
+ \bigl( e_{(1, 2)} \ot \ga_{(1, 2)} \bigr)
\\ & \quad
+ \bigl( e_{(0, 1, 2)} \ot \si \bigr)
\end{aligned}
\end{equation}
where
\begin{equation} \label{eqn:365}
\begin{aligned}
&
f_{(0)} , f_{(1)} , f_{(2)} \in \opn{Hom}_{\Z}(A, B)^{0}
\\ &
\ga_{(0, 1)}, \ga_{(0, 2)}, \ga_{(1, 2)}   \in \opn{Hom}_{\Z}(A, B)^{-1}
\\ &
\si = \si_{(0, 1, 2)}   \in \opn{Hom}_{\Z}(A, B)^{-2} .
\end{aligned}
\end{equation}
The simplicial relations between the coefficients appearing in (\ref{eqn:360})
are shown in diagram (\ref{eqn:392}), where we write
\begin{equation} \label{eqn:430}
g_{(i, j)} = \bmat{f_{(i)} & \bmsp y \cd \ga_{(i, j)} \\ 0 & f_{(j)}} .
\end{equation}

There are surjective DG ring homomorphisms
\[ \pa_0,  \pa_1, \pa_2 : \opn{Cyl}_{2}(B) \to \opn{Cyl}_{1}(B) \]
coming from the simplicial structure.
They induce  degree $0$ homomorphism of graded abelian groups
\begin{equation} \label{eqn:401}
g_{(j, k)} := \pa_i \circ h : A \to \opn{Cyl}_{1}(B)
\end{equation}
for $\{ i, j, k \} = \{ 0, 1, 2 \}$ and $j < k$.
The explicit formulas are
\begin{equation} \label{eqn:351}
g_{(j, k)} = \pa_i \circ h =
\bigl( e_{(j)} \ot f_{(j)} \bigr) +
\bigl( e_{(k)} \ot f_{(k)} \bigr)
+ \bigl( e_{(j, k)} \ot \ga_{(j, k)} \bigr) .
\end{equation}
In matrix notation they are
(\ref{eqn:430}).

The product DG ring homomorphism
\begin{equation} \label{eqn:352}
\pa_0 \times \pa_1 \times \pa_2 : \opn{Cyl}_{2}(B) \to
\opn{Cyl}_{1}(B) \times \opn{Cyl}_{1}(B) \times \opn{Cyl}_{1}(B)
\end{equation}
has kernel
\begin{equation} \label{eqn:353}
\opn{Ker}(\pa_0 \times \pa_1 \times \pa_2) =
e_{(0, 1, 2)} \ot B \sub \opn{Cyl}_{2}(B)  .
\end{equation}

\begin{thm} \label{thm:390}
Consider the degree $0$ homomorphism  of graded abelian groups
$h : A \to \opn{Cyl}_2(B)$ from equation \tup{(\ref{eqn:360})}.
Assume that
$\pa_i \circ h : A \to \opn{Cyl}_{1}(B)$ are DG ring homomorphisms for all
$i$. Then the following two conditions are equivalent\tup{:}
\begin{itemize}
\rmitem{i} $h$ is a DG ring homomorphism.

\rmitem{ii} These equalities hold\tup{:}
\[ \tag*{\hspace{1cm} ($*$)}
\si(a \cd a') =
\si(a) \cd f_{(2)}(a') + f_{(0)}(a) \cd \si(a')
+ (-1)^{k - 1} \cd \ga_{(0, 1)}(a) \cd \ga_{(1, 2)}(a') \]
for all $a \in A^k$ and $a' \in A^{k'}$, and
\[ \tag*{\hspace{0.8cm} ($**$)}
\si \circ \d_A - \d_{B} \circ \si =
\ga_{(0, 1)} - \ga_{(0, 2)} + \ga_{(1, 2)} . \]
\end{itemize}
\end{thm}

\begin{proof}
(i) $\Rightarrow$ (ii):
To verify equality ($*$) we are going to calculate the coefficients of
$e_{(0, 1, 2)}$ in $h(a \cd a')$ and $h(a) \cd h(a')$. The first is easy:
\begin{equation} \label{eqn:376}
h(a \cd a') = \cdots + \bigl( e_{(0, 1, 2)} \ot \si(a \cd a') \bigr) .
\end{equation}
As for the other:
\begin{equation} \label{eqn:377}
\begin{aligned}
&
h(a) \cd h(a') =
\bigl( e_{(0)} \ot f_{(0)}(a)
+  e_{(0, 1)} \ot \ga_{(0, 1)}(a)
+ e_{(0, 1, 2)} \ot \si(a) + \cdots \bigr)
\\ & \quad
\cd \bigl(  e_{(2)} \ot f_{(2)}(a')
+   e_{(1, 2)} \ot \ga_{(1, 2)}(a')
+  e_{(0, 1, 2)} \ot \si(a') + \cdots \bigr)
\\ &
= e_{(0, 1, 2)} \ot \bigl( (-1)^{k \cd 2} \cd f_{(0)}(a) \cd \si(a')
+  (-1)^{k - 1} \cd \ga_{(0, 1)}(a) \cd \ga_{(1, 2)}(a')
\\ & \quad + \si(a) \cd f_{(2)}(a') \bigr) + \cdots
\\ &
= e_{(0, 1, 2)} \ot \bigl( f_{(0)}(a) \cd \si(a')
+  (-1)^{k - 1} \cd \ga_{(0, 1)}(a) \cd \ga_{(1, 2)}(a')
\\ & \quad
+ \si(a) \cd f_{(2)}(a') \bigr) + \cdots
\end{aligned}
\end{equation}
Since $h(a \cd a') = h(a) \cd h(a')$, we conclude that ($*$) holds.

To verify equality ($**$) we shall calculate the coefficients of
$e_{(0, 1, 2)}$ in $\d_{\opn{Cyl}} \circ h$ and $h \circ \d_A$.
Here are the calculations
\begin{equation} \label{eqn:370}
\begin{aligned}
&
(\d_{\opn{Cyl}} \circ h)(a)
\\ & \quad
= \d_{\opn{Cyl}} \bigl( \cdots +  \bigl( e_{(0, 1)} \ot \ga_{(0, 1)}(a) \bigr)
+ \bigl( e_{(0, 2)} \ot \ga_{(0, 2)}(a) \bigr)
+ \bigl( e_{(1, 2)} \ot \ga_{(1, 2)}(a) \bigr)
\\ & \quad
\quad + \bigl( e_{(0, 1, 2)} \ot \si(a) \bigr) \bigr)
\\ & \quad
= e_{(0, 1, 2)} \ot
\bigl( \ga_{(0, 1)}(a) - \ga_{(0, 2)}(a) + \ga_{(1, 2)}(a) +
\d_{B}(\si(a)) \bigr) + \cdots .
\end{aligned}
\end{equation}
and
\begin{equation} \label{eqn:371}
(h \circ \d_A)(a) = \cdots + e_{(0, 1, 2)} \ot \si(\d_A(a)) .
\end{equation}
Since
$\d_{\opn{Cyl}} \circ h = h \circ \d_A$,
we deduce the equality $(**)$.

\bigskip \noindent
(ii) $\Rightarrow$ (i):
We know that each
$\pa_i \circ h : A \to \opn{Cyl}_{1}(B)$
is a DG ring homomorphism. Hence
\[ (\pa_0 \times \pa_1 \times \pa_2) \circ h : A \to
\opn{Cyl}_{1}(B) \times \opn{Cyl}_{1}(B) \times \opn{Cyl}_{1}(B) \]
is a DG ring homomorphism.

Equation (\ref{eqn:353}) implies that $h$ is a graded ring homomorphism iff
the coefficients of $e_{(0, 1, 2)}$ in
$h(a \cd a')$ and $h(a) \cd h(a')$ are equal.
The calculations in (\ref{eqn:377}) and (\ref{eqn:376}), with the equality
($*$) in condition (ii), validate this assertion.

Similarly, $h$ is a strict homomorphism of DG abelian groups, i.e.\
$h \circ \d_A = \d_{\mrm{cyl}} \circ h$, iff the
coefficients of $e_{(0, 1, 2)}$ in
$(\d_{\opn{Cyl}} \circ h)(a)$ and $(h \circ \d_A)(a)$ are equal.
This is true by formulas (\ref{eqn:370}) and (\ref{eqn:371}), and
equality ($**$) in condition (ii).
\end{proof}

\section{Properties of the Hom Groupoid}
\label{sec:props-hom-grpd}

Given a semi-free DG ring $\til{A}$ and an arbitrary DG ring $B$,  the
simplicial set
\begin{equation} \label{eqn:435}
\opn{SHom}_{}(\til{A}, B) =
\{ \opn{Hom}_{\cat{DGRng}}(A, \opn{Cyl}_{q}(B)) \}_{q \in \N}
\end{equation}
is a Kan complex, by Theorem \ref{thm:215}. Its fundamental groupoid
\begin{equation} \label{eqn:436}
\opn{SHom}_{\leq 1}(\til{A}, B) =
\bpi_{\leq 1} (\opn{SHom}_{}(\til{A}, B))
\end{equation}
is called the Hom groupoid from $\til{A}$ to $B$.
This is Definition \ref{dfn:400}.
There is a simplicial description of this groupoid at the end of Section
\ref{sec:kan-cond}.

In this section we prove Theorems \ref{thm:270} and \ref{thm:271} 
from the Introduction, on some properties of the Hom groupoid,
repeated here as Theorems \ref{thm:302} and \ref{thm:303}, respectively.

\begin{thm} \label{thm:302}
Let $\til{A}$ be a semi-free DG ring, and let $v : \til{B} \to B$ be a
surjective quasi-isomorphism of DG rings. Then the map of groupoids
\[ \opn{SHom}_{\leq 1}(\til{A}, \til{B}) \to
\opn{SHom}_{\leq 1}(\til{A}, B) \]
induced by $v$ is a surjective equivalence.
\end{thm}

\begin{proof}
Let us write
\[ G := \opn{SHom}_{\leq 1}(\opn{id}_{\til{A}}, v) :
\opn{SHom}_{\leq 1}(\til{A}, \til{B}) \to \opn{SHom}_{\leq 1}(\til{A}, B) . \]
We need to prove that $G$ is a surjective equivalence of groupoids.

\medskip \noindent
Step 1. Here we prove the surjectivity of $G$ on objects.
Take an object $f$ of $\opn{SHom}_{\leq 1}(\til{A}, B)$.
This means that $f : \til{A} \to B$ is a DG ring homomorphism.
Since $v$ is a surjective quasi-isomorphism, according to Theorem
\ref{thm:157} there is a lifting
$\til{f} : \til{A} \to \til{B}$ of $f$.
In terms of the groupoids that says that
$f = G(\til{f})$.

\medskip \noindent
Step 2. Now we prove the surjectivity of $G$ on morphisms, i.e.\ that $G$ is a
full functor.
Take objects $\til{f}_0$ and $\til{f}_1$ in
$\opn{SHom}_{\leq 1}(\til{A}, \til{B})$,
and define
$f_i := G(\til{f}_i) = v \circ \til{f}_i$,
which are objects of $\opn{SHom}_{\leq 1}(\til{A}, B)$.
Suppose
$[g] : f_0 \to f_1$ is an isomorphism in the groupoid
$\opn{SHom}_{\leq 1}(\til{A}, B)$,
represented by $g \in \opn{SHom}_{1}(\til{A}, B)$.
The DG ring homomorphism $g : \til{A} \to \opn{Cyl}_{1}(B)$ corresponds to a
Keller homotopy
$\ga : f_0 \twoto f_1$, see Proposition \ref{prop:330}.

Consider the intermediate DG ring
\begin{equation} \label{eqn:316}
D := \bigl( \opn{Cyl}_{1}(\Z)^0 \ot_{\Z} \til{B} \bigr) \oplus
\bigl( \opn{Cyl}_{1}(\Z)^1 \ot_{\Z} B \bigr) ,
\end{equation}
as in Proposition \ref{prop:425}. In matrix notation it looks like this:
\[ D = \bmat{\til{B} & B[-1] \\ 0 & \til{B}}
= \bmat{1 \ot\til{B} & y \ot B \\ 0 & 1 \ot \til{B} }  , \]
cf.\ Definition \ref{dfn:332}.
There is a commutative diagram of DG rings
\[ \begin{tikzcd} [column sep = 8ex, row sep = 6ex]
\opn{Cyl}_{1}(\til{B})
\ar[rr, bend left = 20, "{\opn{Cyl}_{1}(v))}"]
\ar[r, "{\til{u}}" swap]
&
D
\ar[r, "{u}" swap]
&
\opn{Cyl}_{1}(B)
\end{tikzcd} \]
in which
\begin{equation} \label{eqn:321}
\til{u} :=
\bmat{\opn{id}_{\til{B}} & v[-1] \\ 0 & \opn{id}_{\til{B}}}
\quad \text{and} \quad
u := \bmat{v & \opn{id}_{B}[-1] & \\ 0 & v} .
\end{equation}
All three homomorphisms are surjective quasi-isomorphisms.

Define the intermediate DG ring homomorphism
\begin{equation} \label{eqn:315}
g^{D} : \til{A} \to D, \bmsp
g^{D} :=
\bmat{\til{f}_0 & y \ot \ga \\ 0 & \til{f}_1} .
\end{equation}
In terms of equation (\ref{eqn:316}), and the basis
$e_{(0)}$, $e_{(1)}$ and $e_{(0, 1)}$ of $\opn{Cyl}_{1}(\Z)$,
we can express $g^{D}$ as a sum
\begin{equation} \label{eqn:317}
g^{D} = (e_{(0)} \ot \til{f}_0) +
(e_{(1)} \ot \til{f}_1) + (e_{(0, 1)} \ot \ga) .
\end{equation}
The homomorphism $g^{D}$ satisfies
$u \circ g^{D} = g$.
Since $\til{u}$ is a surjective quasi-isomorphism, and $\til{A}$ is semi-free,
by Theorem \ref{thm:157} the homomorphism
$g^{D}$ lifts to a homomorphism
$\til{g} : \til{A} \to \opn{Cyl}_{1}(\til{B})$
such that
$\til{u} \circ \til{g} =  g^{D}$.
See the commutative diagram of DG rings
\[ \begin{tikzcd} [column sep = 10ex, row sep = 5ex]
\til{A}
\ar[r, dashed, "{\til{g}}"]
\ar[dr, "{g^{D}}", swap]
&
\opn{Cyl}_{1}(\til{B})
\ar[d, "{\til{u}}"]
\\
&
D
\end{tikzcd} \]
In matrix notation, $\til{g}$ looks like this:
\begin{equation} \label{eqn:431}
\til{g} =
\bmat{\til{f}_0 & y \ot \til{\ga} \\ 0 & \til{f}_1} ,
\end{equation}
where
$\til{\ga} : \til{A} \to \til{B}$
is a degree $-1$ homomorphism of graded abelian groups.
By Proposition \ref{prop:330}, $\til{\ga}$ is a Keller homotopy
$\til{\ga} : \til{f}_0 \twoto \til{f}_1$.
Thus $[\til{g}]$ is an isomorphism from $\til{f}_0$ to $\til{f}_1$
in the groupoid $\opn{SHom}_{\leq 1}(\til{A}, \til{B})$,
and $G([\til{g}]) = [g]$.

\medskip \noindent
Step 3. Now we prove the injectivity of $G$ on morphisms, i.e.\ that $G$ is a
faithful functor. For groupoids it is enough to verify injectivity on
automorphism groups. Namely, it suffices to prove that for every object
$\til{f}$ in $\opn{SHom}_{\leq 1}(\til{A}, \til{B})$,
with automorphism group $\opn{Aut}(\til{f})$, and with image
$f$ in $\opn{SHom}_{\leq 1}(\til{A}, B)$, the group homomorphism
$G : \opn{Aut}(\til{f}) \to \opn{Aut}(f)$ is injective.

Take elements
$[\til{g}], [\til{g}'] \in \opn{Aut}(\til{f})$,
with images
$[g], [g'] \in \opn{Aut}(f)$.
We use the usual matrix notation for $\til{g}, \til{g}', \ldots$,
but now with decorations, e.g.\
$\til{g}' =
\sbmat{\til{f} & y \ot \til{\ga}' \\ 0 & \til{f}}$.
Suppose that $[g] = [g']$ in $\opn{Aut}(f)$. This is the same as saying
that $[g] = [g']$ as morphisms in $\opn{SHom}_{\leq 1}(\til{A}, B)$.
By the definition of simplicial homotopy, see
\cite[Definition 1.4.3.1, tag =
\href{https://kerodon.net/tag/003V}{\tt 003V}]{Lu2}
and the discussion at the end of Section
\ref{sec:kan-cond}, there exists an element
$h \in \opn{SHom}_{2}(\til{A}, B)$
satisfying
$\pa_0(h) = \opn{id}_{f}$,
$\pa_1(h) = g$ and
$\pa_2(h) = g'$ in
$\opn{SHom}_{1}(\til{A}, B)$.
Note that
$\opn{id}_{f} = \opn{s}_0(f) =
\sbmat{f & \bmsp 0 \\ 0 & f}$.

The element $h$ is a DG ring homomorphism
$h : \til{A} \to \opn{Cyl}_{2}(B)$.
We can't express $h$ as matrix; but, like in formula (\ref{eqn:360}), we
can express $h$ as a sum
\begin{equation} \label{eqn:320}
\begin{aligned}
&
h =
(e_{(0)} \ot f) + (e_{(1)} \ot f) + (e_{(2)} \ot f)
\\ & \quad
+ (e_{(0, 1)} \ot \ga') + (e_{(0, 2)} \ot \ga)
+ (e_{(1, 2)} \ot 0)
\\ & \quad
+ (e_{(0, 1, 2)} \ot \si)  ,
\end{aligned}
\end{equation}
where
$\ga, \ga' \in \opn{Hom}_{\Z}(\til{A}, B)^{-1}$
and
$\si \in \opn{Hom}_{\Z}(\til{A}, B)^{-2}$.

Define the new intermediate DG ring $D$ to be
\begin{equation} \label{eqn:322}
D :=
\bigl( \opn{Cyl}_{2}(\Z)^0 \ot_{\Z} \til{B} \bigr)
\oplus
\bigl( \opn{Cyl}_{2}(\Z)^1 \ot_{\Z} \til{B} \bigr)
\oplus
\bigl( \opn{Cyl}_{2}(\Z)^2 \ot_{\Z} B \bigr)
\end{equation}
as in {Proposition} \ref{prop:425}, for the direct system of DG rings
$\til{B} \xar{\opn{id}} \til{B} \xar{v} B$.
In terms of basis elements, we have
\begin{equation} \label{eqn:324}
\begin{aligned}
&
D =
(e_{(0)} \ot \til{B}) \oplus (e_{(1)} \ot \til{B}) \oplus (e_{(2)} \ot \til{B})
\\ & \quad
\oplus (e_{(0, 1)} \ot \til{B}) \oplus (e_{(0, 2)} \ot \til{B})
\oplus (e_{(1, 2)} \ot \til{B})
\\ & \quad
\oplus (e_{(0, 1, 2)} \ot B) .
\end{aligned}
\end{equation}
There is a commutative diagram of DG rings
\[ \begin{tikzcd} [column sep = 8ex, row sep = 6ex]
\opn{Cyl}_{2}(\til{B})
\ar[rr, bend left = 20, "{\opn{Cyl}_{2}(v))}"]
\ar[r, "{\til{u}}" swap]
&
D
\ar[r, "{u}" swap]
&
\opn{Cyl}_{2}(B)
\end{tikzcd} \]
in which $\til{u}$ and $u$ are by now obvious (cf.\ equation (\ref{eqn:321})).

Define the intermediate graded abelian group homomorphism
$h^{D} : \til{A} \to D$ to be
\begin{equation} \label{eqn:323}
\begin{aligned}
&
h^{D} :=
(e_{(0)} \ot \til{f}) + (e_{(1)} \ot \til{f}) + (e_{(2)} \ot \til{f})
\\ & \quad
+ (e_{(0, 1)} \ot \til{\ga}') + (e_{(0, 2)} \ot \til{\ga})
+ (e_{(1, 2)} \ot 0)
\\ & \quad
+ (e_{(0, 1, 2)} \ot \si) .
\end{aligned}
\end{equation}
We claim that $h^{D}$ is a DG $A$-ring homomorphism.

By construction we have an equality
$u \circ h^{D} = h$ of graded abelian group homomorphisms $\til{A} \to
\opn{Cyl}_2(B)$, see formulas (\ref{eqn:320}) and (\ref{eqn:323}).
Here is the commutative diagram displaying it:
\begin{equation} \label{eqn:451}
\begin{tikzcd} [column sep = 10ex, row sep = 5ex]
\til{A}
\ar[r, "{h^D}"]
\ar[dr, "{h}", swap]
&
D
\ar[d, "{u}"]
\\
&
\opn{Cyl}_{2}(B)
\end{tikzcd}
\end{equation}
Recall that $h$ is a DG ring homomorphism.
We also have the equalities
\[ \pa_0 \circ h^D = \opn{id}_{\til{f}} =
\bmat{\til{f} & 0 \\ 0 & \til{f}} , \
\pa_1 \circ h^D = \til{g} =
\bmat{\til{f} & y \ot \til{\ga} \\ 0 & \til{f}} , \
\pa_2 \circ h^D = \til{g}' =
\bmat{\til{f} & y \ot \til{\ga}' \\ 0 & \til{f}} . \]
These are DG $A$-ring homomorphisms $\til{A} \to \opn{Cyl}_{1}(B)$.
Like in the proof of the implication (ii) $\twoto$ (i) in Theorem
\ref{thm:390}, it remains to prove that $\si$ satisfies equations ($*$) and
($**$) there, applied to $h^{D}$.

From formula (\ref{eqn:320}) we see that $\si$ is also the
coefficient of $e_{(0, 1, 2)}$ in the DG ring homomorphism $h$.
The implication (i) $\twoto$ (ii) in Theorem
\ref{thm:390}, applied to $h$, tells us that $\si$ satisfies equations ($*$) and
($**$) there. Conclusion: $h^D$ is a DG $A$-ring homomorphism.

Since $\til{u}$ is a surjective quasi-isomorphism, and $\til{A}$ is semi-free,
by Theorem \ref{thm:157} the DG $A$-ring homomorphism
$h^{D}$ lifts to a homomorphism
$\til{h} : \til{A} \to \opn{Cyl}_{2}(\til{B})$
such that
$\til{u} \circ \til{h}^{} =  h^{D}$.
See the commutative diagram of DG rings
\[ \begin{tikzcd} [column sep = 10ex, row sep = 5ex]
\til{A}
\ar[r, dashed, "{\til{h}}"]
\ar[dr, "{h^{D}}", swap]
&
\opn{Cyl}_{2}(\til{B})
\ar[d, "{\til{u}}"]
\\
&
D
\end{tikzcd} \]
In coordinates we have
\begin{equation} \label{eqn:325}
\begin{aligned}
&
\til{h}_{} =
(e_{(0)} \ot \til{f}) + (e_{(1)} \ot \til{f}) + (e_{(2)} \ot \til{f})
\\ & \quad
+ (e_{(0, 1)} \ot \til{\ga}') + (e_{(0, 2)} \ot \til{\ga})
+ (e_{(1, 2)} \ot 0)
\\ & \quad
+ (e_{(0, 1, 2)} \ot \til{\si}) ,
\end{aligned}
\end{equation}
where
$\til{\si} : \til{A} \to \til{B}$
is a degree $-2$ homomorphism of graded abelian groups.
The element $\til{h}$ belongs to
$\opn{SHom}_{2}(\til{A}, \til{B})$.
It satisfies
$\pa_0(\til{h}) = \opn{id}_{\til{f}}$,
$\pa_1(\til{h}) = \til{g}$ and
$\pa_2(\til{h}) = \til{g}'$
in $\opn{SHom}_{1}(\til{A}, \til{B})$.
Therefore
$[g] = [g']$ as isomorphisms in
$\opn{SHom}_{\leq 1}(\til{A}, B)$;
or in other words, as elements of $\opn{Aut}(\til{f})$.
\end{proof}

\begin{rem} \label{rem:385}
Theorem \ref{thm:302} can probably be made a bit stronger: if $v$ is a
quasi-iso\-morphism (not necessarily surjective) then
$\opn{SHom}_{\leq 1}(\opn{id}_{\til{A}}, v)$
is an equivalence.
\end{rem}

To finish, here is a copy of Theorem \ref{thm:271}.

\begin{thm} \label{thm:303}
Let $f : A \to B$ be a DG ring homomorphism, let
$u : \til{A} \to A$ and $v : \til{B} \to B$
be semi-free resolutions, and let
$\{ \til{f}_i \}_{i \in I}$ be the collection of all liftings
$\til{f}_i : \til{A} \to \til{B}$ of $f$ with respect to $u$ and $v$.
Then\tup{:}
\begin{enumerate}
\item For every $i, j \in I$ there is a distinguished isomorphism
$[\til{g}_{i, j}] : \til{f}_i \to \til{f}_j$ in the groupoid
$\opn{SHom}_{\leq 1}(\til{A}, \til{B})$.

\item The collection of distinguished isomorphisms
$\{ [\til{g}_{i, j}] \}_{i, j \in I}$
satisfies
$[\til{g}_{j, k}] \cd [\til{g}_{i, j}] = [\til{g}_{i, k}]$
for all $i, j, k \in I$, and  $[\til{g}_{i, i}] = \opn{id}_{\til{f}_i}$.
\end{enumerate}
\end{thm}

Here is the commutative diagram depicting the liftings:
\begin{equation} \label{eqn:437}
\begin{tikzcd} [column sep = 12ex, row sep = 6ex]
\til{A}
\ar[d, "{u}" swap]
\ar[r, dashed, "{\til{f}^{}_{i}}"]
\ar[dr, "{f \circ u}"]
&
\til{B}
\ar[d, "{v}"]
\\
A
\ar[r, "{f}" swap]
&
B
\end{tikzcd}
\end{equation}

\begin{proof}
Let's introduce the abbreviation
$S(\til{A}, \til{B}) := \opn{SHom}_{}(\til{A}, \til{B})$
for this Kan simplicial set. With this notation, the Hom groupoid is
$S_{\leq 1}(\til{A}, \til{B})$.

The liftings we are looking at are the DG ring homomorphisms
$\til{f}_i : \til{A} \to \til{B}$ such that
$v \circ \til{f}_i = f \circ u$.
See diagram (\ref{eqn:437}).
By Theorem \ref{thm:302} there is a bijection
\begin{equation} \label{eqn:385}
\opn{S}_{\leq 1}(\til{A}, \til{B})(\til{f}_i, \til{f}_j)
\iso \opn{S}_{\leq 1}(\til{A}, B)(f \circ u, f \circ u)
\end{equation}
induced by $v$.
Define
$[\til{g}_{i, j}] : \til{f}_i \to \til{f}_j$
to be the unique isomorphism in the groupoid
$\opn{S}_{\leq 1}(\til{A}, \til{B})$
that goes by (\ref{eqn:385}) to
$\opn{id}_{f \circ u} \in \opn{S}_{\leq 1}(\til{A}, B)$.

Given indices $i, j, k \in I$, the composed isomorphism
$[\til{g}_{j, k}] \cd [\til{g}_{i, j}]$ in
$\opn{S}_{\leq 1}(\til{A}, \til{B})(\til{f}_i, \til{f}_k)$
goes by (\ref{eqn:385}) to $\opn{id}_{f \circ u}$.
Therefore, by the uniqueness, there is equality
$[\til{g}_{j, k}] \cd [\til{g}_{i, j}] = [\til{g}_{i, k}]$.
Likewise for $[\til{g}_{i, i}]$.
\end{proof}


\end{document}